\documentclass[a4paper,11pt]{article}
\usepackage{inputenc}
\usepackage{authblk}
\usepackage[table]{xcolor}
\usepackage{color}
\usepackage{amsmath}
\usepackage{amssymb}
\usepackage{amsthm}
\usepackage{booktabs}
\usepackage{xcolor,graphicx,float}
\usepackage{hyperref}
\usepackage{soul}
\usepackage{indentfirst}
\hypersetup{colorlinks=true,
	linkcolor=blue,
	anchorcolor=blue,
	citecolor=blue}
\usepackage{enumitem}
\setlist[itemize]{itemsep=0pt, topsep=0pt}

\theoremstyle{plain}
\newtheorem{theorem}{\bf Theorem}[section]
\newtheorem{lemma}[theorem]{\bf Lemma}
\newtheorem{proposition}[theorem]{\bf Proposition}

\theoremstyle{remark}

\newtheorem{remark}{\bf Remark}
\newtheorem{construction}{\bf Construction}

\numberwithin{equation}{section}

\newcommand{\spn}[2]{\genfrac{\{}{\}}{0pt}{}{#1}{#2}}
\title{\bf On $r$-cross $t$-intersecting families of partitions}
\author[1]{Jie Wen\thanks{E-mail: \text{jwen@mail.bnu.edu.cn}}}
\author[1]{Benjian Lv\thanks{Corresponding author. E-mail: \text{bjlv@bnu.edu.cn}}}
\affil[1]{\small Laboratory of Mathematics and Complex Systems (Ministry of Education), School of Mathematical Sciences, Beijing Normal University, Beijing 100875, China}
\date{}

\begin{document}
	\renewcommand{\baselinestretch}{1.2}
	\maketitle
	\begin{abstract}
In this paper, we address several intersection problems for $r$-cross $t$-intersecting families of partitions. A  $k$-partition of an $n$-set $X$ is a set of $k$ pairwise disjoint non-empty subsets whose union is $X$. For $1\leq i\leq r$, let $\mathcal{F}_i$ be a family of $k_i$-partitions of $X$. We say that $\mathcal{F}_1,\mathcal{F}_2,\ldots,\mathcal{F}_r$ are $r$-cross $t$-intersecting if $|\cap_{i=1}^{r}F_i|\geq t$ for all $F_i\in\mathcal{F}_i$. The families are called non-trivial if $|\cap_{i=1}^r(\cap_{F\in\mathcal{F}_i}F)|<t$. Proving an Erd\H{o}s–Ko–Rado type theorem, we determine the families maximizing $\prod_{i=1}^r|\mathcal{F}_i|$. We further determine non-trivial $r$-cross $t$-intersecting families with maximum product of sizes; this result also serves as a  Hilton–Milner type theorem. In particular, for $r=2$ there are two potential structures for optimal families, and for $r\geq3$ exactly one remains.	
	
		
		
	\end{abstract}
	\section{Introduction}
We begin with some standard notations. Throughout this paper, lowercase letters denote positive integers, unless otherwise stated. 
We set $[n]:=\{1,2,\ldots,n\}$, and set $[i,j]=\{i,i+1,\ldots,j\}$ for $i\leq j$.  Denote by $\binom{[n]}{k}$ the family of all $k$-subsets of $[n]$. For a family $\mathcal{F}$ of sets, we use  $\cup\mathcal{F}$ and $\cap\mathcal{F}$ to denote the union and the intersection of its members, respectively. We write  $\overline{A}=[n]\setminus A$ for $A\subseteq[n]$.

Originating from the celebrated  Erd\H{o}s–Ko–Rado theorem \cite{Erdos-Ko-Rado-1961}, intersection problems have evolved into a longstanding and active direction of research in extremal set theory. A \emph{$t$-intersecting} family in $\binom{[n]}{k}$ is one whose pairwise intersections are all have size at least $t$. The theorem states that there exists $f(k,t)$ such that, for all $n\geq f(k,t)$, every $t$-intersecting family has size at most $\binom{n-t}{k-t}$. Moreover, if $n>f(k,t)$, then  $\left\{F\in\binom{[n]}{k}:[t]\subseteq F\right\}$ is, up to a permutation on $[n]$, the unique family achieving the upper bound. It is well-known that the least possible value of $f(k,t)$ equals $(t+1)(k-t+1)$. This is established in the original paper \cite{Erdos-Ko-Rado-1961} for $t=1$. Frankl proved this in the breakthrough paper \cite{Frankl-1976} for $t\geq 15$, and subsequently Wilson \cite{Wilson-1984} determined the exact value for $2\leq t\leq14$, with an ingenious proof valid for all $t$. Another classical line of research concerns the study of large $t$-intersecting families. A family $\mathcal{F}\subseteq\binom{[n]}{k}$ is called \emph{non-trivial} if $|\cap\mathcal{F}|<t$. The study of non-trivial $t$-intersecting families was initiated by Hilton and Milner \cite{Hilton-Milner-1967}, who determined the largest such families for $t=1$. For general $t\geq2$ and large $n$, the problem is settled by Frankl \cite{Frankl-1976}. In \cite{Frankl-1987}, Frankl established  a far-reaching sharpening of the Hilton–Milner theorem, characterizing largest $1$-intersecting families with conditions on the maximum degree. Ahlswede and Khachatrian \cite{Ahlswede-Khachatrian-1996} completely determined maximum-sized non-trivial $t$-intersecting families in $\binom{[n]}{k}$ for all $n,k$ and $t$.

For $r\geq2$ and $k_1,k_2,\ldots,k_r\geq t$, the families $\mathcal{F}_1\subseteq\binom{[n]}{k_1}$, $\mathcal{F}_2\subseteq\binom{[n]}{k_2},\ldots$, $\mathcal{F}_r\subseteq\binom{[n]}{k_r}$ are said to be \emph{$r$-cross $t$-intersecting} if $|F_1\cap\cdots\cap F_r|\geq t$ for all $F_i\in\mathcal{F}_i,1\leq i\leq r$. This serves as a natural generalization for the notion of `$t$-intersecting'. There is a vast and excellent literature on investigating $r$-cross $t$-intersecting families with large product or sum of sizes. To limit the scope of this paper, we restrict our attention to the product version, and refer the readers to \cite{Frankl-2024,Frankl-Tokushige-1998,Gupta et al.-2023,Kupavskii-Zakharov-2018,Wang-Zhang,Wu-Feng-Li-2025} for the sum version.

Let  $\mathcal{F}_1\subseteq\binom{[n]}{k_1}$, $\mathcal{F}_2\subseteq\binom{[n]}{k_2},\ldots$, $\mathcal{F}_r\subseteq\binom{[n]}{k_r}$ be $r$-cross $t$-intersecting. Pyber \cite{Pyber} established a product version of the  Erd\H{o}s-Ko-Rado theorem, that is, for $r=2$ and $t=1$,  $|\mathcal{F}_1||\mathcal{F}_2|\leq\binom{n-1}{k_1-1}\binom{n-1}{k_2-1}$ holds when $n$ is large. Matsumoto and Tokushige \cite{Matsumoto-Tokushige-1989} found  the least possible $n$ for which the upper bound holds. In \cite{Frankl-Tokushige-2011}, Frankl and Tokushige proved that for $t=1$, $k_1=\cdots=k_r=k$ and the natural restriction $n\geq rk/(r-1)$,  $\prod_{i=1}^{r}|\mathcal{F}_i|\leq\binom{n-1}{k-1}^r$ holds, with an elegant proof via Katona's Circle (see also \cite{Frankl-2021}). For general $t\geq2$, $r=2$ and $k_1=k_2=k$, Tokushige \cite{Tokushige-2013} conjectured that   $|\mathcal{F}_1||\mathcal{F}_2|\leq\binom{n-t}{k-t}^2$ for $n\geq(t+1)(k-t+1)$, and if $n>(t+1)(k-t+1)$, then equality holds only if both $\mathcal{F}_1$ and  $\mathcal{F}_2$ consist of $k$-subsets containing $t$ fixed elements. After a series of  contributions \cite{Tokushige-2010,Tokushige-2013,Frankl et al.-2014,Borg-2016} toward the conjecture, it was completely settled very recently by Zhang and Wu \cite{Zhang-Wu-2025} for $t\geq3$, and by Tanaka and Tokushige \cite{Tanaka-Tokushige-2025} for $t=2$.

Besides studying the maximum product, extensive works have focused on describing the structure of large $r$-cross $t$-intersecting families, in the sense that they have a large product of sizes. Frankl and Kupavskii  \cite{Frankl-Kupavskii-2017} generalized Pyber's theorem by proving a size-sensitive inequality for $2$-cross $1$-intersecting families. Other  works mainly considered those families under the natural condition that a common $t$-subset is forbidden, including  $|\cap_{i=1}^r(\cap\mathcal{F}_i)|<t$ (cf. \cite{Cao-Lu-Lv-Wang-2024,He-Li-Wu-Zhang-2026}) or  $\max\{|\cap\mathcal{F}_i|:i\in[r]\}<t$ (cf. \cite{Frankl-Wang-2023,Frankl-Wang-2024}). 

Intersection problems arise naturally for many objects, and their study has inspired a variety of insightful results. For systematic introduction on intersection problems, we refer the readers to survey papers \cite{Deza-Frankl-1983,Frankl-Tokushige-2016} and monographs \cite{Frankl-Tokushige-book,Godsil-Meagher-book}. 

The aim of the present paper is to study intersection problems for partitions. A \emph{partition} is a set of pairwise disjoint non-empty subsets, and the members of a partition are called \emph{blocks}. A \emph{$k$-partition of $M$} is a partition consisting of $k$ blocks whose union equals $M$. We use $\spn{[n]}{k}$ to denote the set of all $k$-partitions of $[n]$. This notation arises from the Stirling partition number
\begin{equation}\label{equspn}
	\spn{n}{k}=\frac{1}{k!}\sum_{j=0}^{k}(-1)^j\binom{k}{j}(k-j)^n.
\end{equation}
Note that $\spn{[n]}{k}$ has size $\spn{n}{k}$. Set $\spn{n}{a}=0$ for $a\leq0$. Two families $\mathcal{G}_1,\mathcal{G}_2\subseteq\spn{[n]}{k}$ are \emph{isomorphic}, denoted $\mathcal{G}_1\cong\mathcal{G}_2$, if they are the same up to some permutation on $[n]$. For a subset $B$ of $[n]$, denote by $[B]:=\{\{i\}:i\in B\}$ the partition of $B$ containing only singletons. A family in $\spn{[n]}{k}$ is said to be \emph{$t$-intersecting} if every two of its members has at least $t$ blocks in common. This coincides with the classical notion in the setting of families of sets. We note that there have been some results on intersection problems for other classes of partitions (cf. e.g., \cite{Meagher-Moura-2005,Ku-Renshaw-2008,Ku-Wong-2012,Ku-Wong-2013,Ku-Wong-2020,Yao-Cao-Zhang}). The following is a $k$-partition version of the Erd\H{o}s-Ko-Rado theorem. Let us set
 \begin{equation}\label{equfunlkt}
	L(k,t):=(t+1)+(k-t+1)\cdot\log_2(t+1)(k-t+1).
\end{equation}
\begin{theorem}[\cite{Erdos-Szekely-2000,Kupavskii-2023,Wen-Lv-2026}]\label{thm0}
Let $k\geq t+2$ and $\mathcal{F}\subseteq\spn{[n]}{k}$ be   $t$-intersecting. If $n\geq L(k,t)$, then 
\begin{equation*}
|\mathcal{F}|\leq\spn{n-t}{k-t}.
\end{equation*}
Moreover, equality holds if and only if  $\mathcal{F}\cong\left\{F\in\spn{[n]}{k}:[[t]]\subseteq F\right\}$.
\end{theorem}
Erd\H{o}s and Sz\'{e}kely \cite{Erdos-Szekely-2000} proved the upper bound for sufficiently large $n$ depending on $k$. Kupavskii \cite{Kupavskii-2023} proved the theorem for $n\geq2k\log_2n$ and $n\geq48$ via the spread method. We note that in \cite{Kupavskii-2023}, the author addressed several Erd\H{o}s–Ko–Rado type questions 
for families of partitions. The previous lower bound of Theorem \ref{thm0} is proved by the authors in \cite{Wen-Lv-2026}. We mention that, by an easy but tedious computation, if $k\sim(1+a)t$ as $t\to\infty$ for a fixed $a>0$, then the least value of $n$ for which Theorem \ref{thm0} holds equals $\Theta(L(k,t))$. In \cite{Wen-Lv-2026}, the authors  also proved the following  Hilton-Milner type theorem for $k$-partitions.
\begin{theorem}[\cite{Wen-Lv-2026}]\label{thmhm}
	Let $k\geq t+3$ and $n\geq2L(k,t)$. Suppose that  $\mathcal{F}\subseteq\spn{[n]}{k}$ is a $t$-intersecting family with $|\cap\mathcal{F}|<t$. The following hold.
	\begin{itemize}
		\item[\rm(i)]If $k\geq2t+3$, then $|\mathcal{F}|\leq\sum_{j=1}^{k-t}(-1)^{j-1}\binom{k-t}{j}\spn{n-t-j}{k-t-j}+t$, with equality precisely if
\begin{align*}
\mathcal{F}\cong&\left\{F\in\spn{[n]}{k}:[[t]]\subseteq F,\;F\cap[[t+1,k]]\neq\emptyset\right\}\\&\cup\{([[k]]\setminus\{i\})\cup\{\{i\}\cup[k+1,n]\}:i\in[t]\}.
					\end{align*}
		\item[\rm(ii)] If $k\leq 2t+2$, then 
		$|\mathcal{F}|\leq(t+2)\spn{n-t-1}{k-t-1}-(t+1)\spn{n-t-2}{k-t-2}$, with equality precisely if $\mathcal{F}\cong\left\{F\in\spn{[n]}{k}:|F\cap[[t+2]]|\geq t+1\right\}$, or $(k,t)=(4,1)$ and $\mathcal{F}$ is isomorphic to the family in {\rm(i)}.
	\end{itemize}
\end{theorem}
We say that $\mathcal{F}_1\subseteq\spn{[n]}{k_1}$, $\mathcal{F}_2\subseteq\spn{[n]}{k_2},\ldots,\mathcal{F}_r\subseteq\spn{[n]}{k_r}$ are \emph{$r$-cross $t$-intersecting} if $|\cap_{i=1}^{r}F_i|\geq t$ for all $F_i\in\mathcal{F}_i,\;i\in[r]$.  Note that if two partitions of $[n]$ with both sizes at most $t+1$ share at least $t$ blocks, then they are identical. Therefore, if  $\max\{k_i:i\in[r]\}\leq t+1$, then such an $r$-tuple exists only if $k_1=\cdots=k_r$, and there is an $F\in\spn{[n]}{k_1}$ such that $\mathcal{F}_1=\cdots=\mathcal{F}_r=\{F\}$. Hence we assume $\max\{k_i:i\in[r]\}\geq t+2$ in the sequel.  Our main results focus on $r$-cross $t$-intersecting families with large product of sizes.  First, the following  Erd\H{o}s–Ko–Rado type theorem determines the families with maximum product of sizes.
\begin{theorem}\label{thm1}
Suppose $r\geq2$, $k_1\geq k_2\geq\cdots\geq k_r\geq t+1$, $k_1\geq t+2$ and $n\geq L(k_1,t)$. If $\mathcal{F}_1\subseteq\spn{[n]}{k_1}$, $\mathcal{F}_2\subseteq\spn{[n]}{k_2},\ldots,\mathcal{F}_r\subseteq\spn{[n]}{k_r}$ are $r$-cross $t$-intersecting, then
$$\prod_{i=1}^r|\mathcal{F}_i|\leq\prod_{i=1}^{r}\spn{n-t}{k_i-t}.$$
Moreover, equality holds if and only if there is a partition $X$ consisting of $t$ singletons such that  $\mathcal{F}_i=\left\{F\in\spn{[n]}{k_i}:X\subseteq F\right\},i=1,2,\ldots,r$.
\end{theorem}
By setting $\mathcal{F}_1=\mathcal{F}_2=\cdots=\mathcal{F}_r$, one deduce directly Theorem \ref{thm0}. An $r$-tuple of $r$-cross $t$-intersecting families $\mathcal{F}_1\subseteq\spn{[n]}{k_1},\ldots,\mathcal{F}_r\subseteq\spn{[n]}{k_r}$ is  \emph{non-trivial} if $|\cap_{i=1}^r(\cap\mathcal{F}_i)|<t$. By a slightly abuse of notation, we say that the families are non-trivial. Otherwise, the families are \emph{trivial}. The families are \emph{maximal} if $\mathcal{F}_i=\mathcal{F}'_i,\;i=1,\ldots,r$ for all $r$-cross $t$-intersecting families $\mathcal{F}'_1,\mathcal{F}'_2,\ldots,\mathcal{F}'_r$ with $\mathcal{F}_i\subseteq\mathcal{F}'_i,\;i=1,\ldots,r$. For simplicity, we use `cross $t$-intersecting' instead of `$2$-cross $t$-intersecting'. The case that $\max\{k_i:i\in[r]\}=t+1$ turns out to be elementary, and we give a complete characterization of all such families; see Propositions \ref{propell=t+1} and \ref{propell=t+1'}. Hence, in what follows, we assume $\max\{k_i:i\in[r]\}\geq t+2$. Our next main result determines non-trivial cross $t$-intersecting families with maximum product of sizes. Before presenting it, let us introduce two constructions.
\begin{construction}
Let $M$ be an $\ell$-partition and let $X\subseteq M$ with $|X|=t$. Define
\begin{align*}
\mathcal{A}(k,X,M)&:=\left\{F\in\spn{[n]}{k}:X\subseteq F,\;F\cap(M\setminus X)\neq\emptyset\right\},\\
\mathcal{B}(\ell,X,M)&:=\left\{F\in\spn{[n]}{\ell}:X\subseteq F\right\}\cup\{(M\setminus\{B\})\cup\{B\cup\overline{\cup M}\}:B\in X\}.
\end{align*}
Set $\mathcal{A}(k,\ell,t):=\mathcal{A}(k,[[t]],[[\ell]])$ and $\mathcal{B}(\ell,t):=\mathcal{B}(\ell,[[t]],[[\ell]])$.
\end{construction}
\begin{construction}
	Let $T$ be a $(t+1)$-partition. Define
	\begin{align*}
		\mathcal{C}(k,T)&:=\left\{F\in\spn{[n]}{k}:T\subseteq F\right\},\\
		\mathcal{D}(\ell,T)&:=\left\{F\in\spn{[n]}{\ell}:|T\cap F|\geq t\right\}.
	\end{align*}
Set $\mathcal{C}(k,t):=\mathcal{C}(k,[[t+1]])$ and $\mathcal{D}(\ell,t):=\mathcal{D}(\ell,[[t+1]])$.
\end{construction}
Two pairs of families of partitions $(\mathcal{F},\mathcal{G})$ and $(\mathcal{F}',\mathcal{G}')$ are \emph{isomorphic}, denoted $(\mathcal{F},\mathcal{G})\cong(\mathcal{F}',\mathcal{G}')$, if they are the same up to some permutation on $[n]$. 
\begin{theorem}\label{thm2}
Let $k_1\geq k_2\geq t+2$, $(k_1,k_2)\neq(3,3)$ and $n\geq2L(k_1,t)$. If  $\mathcal{F}_1\subseteq\spn{[n]}{k_1}$ and $\mathcal{F}_2\subseteq\spn{[n]}{k_2}$ are  non-trivial cross $t$-intersecting families maximizing $|\mathcal{F}_1||\mathcal{F}_2|$,  then $(\mathcal{F}_1,\mathcal{F}_2)$ is isomorphic to one of $(\mathcal{A}(k_1,k_2,t),\mathcal{B}(k_2,t))$, $(\mathcal{B}(k_1,t),\mathcal{A}(k_2,k_1,t))$, $(\mathcal{C}(k_1,t),\mathcal{D}(k_2,t))$ and $(\mathcal{D}(k_1,t),\mathcal{C}(k_2,t))$.
\end{theorem}
We treat the special case $(k_1,k_2)=(3,3)$ separately in Proposition \ref{propkl=33}. We mention that each of the four pairs appearing in Theorem \ref{thm2} can be optimal, depending on the specific values of $n,k_1,k_2$ and $t$. However, it might be hard to explicitly describe the ranges of these parameters for which a given pair has maximum product of sizes. Let us  present an alternative version that includes characterizations of optimal families.
\begin{theorem}\label{thmwholargest}
	Let $k_1\geq k_2\geq t+2$, $n\geq2L(k_1,t)$ and let $\mathcal{F}_1\subseteq\spn{[n]}{k_1}$ and $\mathcal{F}_2\subseteq\spn{[n]}{k_2}$ be non-trivial cross $t$-intersecting families. The following hold.
	\begin{itemize}
		\item[\rm(i)]If $k_2\geq2t+2$ and $n\geq t+1+(k_1-t)(k_2-t)$, then $$|\mathcal{F}_1||\mathcal{F}_2|\leq\left(\sum_{j=1}^{k_2-t}(-1)^{j-1}\binom{k_2-t}{j}\spn{n-t-j}{k_1-t-j}\right)\left(\spn{n-t}{k_2-t}+t\right).$$
		Moreover, equality holds if and  only if  $(\mathcal{F}_1,\mathcal{F}_2)\cong(\mathcal{A}(k_1,k_2,t),\mathcal{B}(k_2,t))$, or $k_1=k_2$ and $(\mathcal{F}_1,\mathcal{F}_2)\cong(\mathcal{B}(k_1,t),\mathcal{A}(k_2,k_1,t))$.
		\item[\rm(ii)]If $k_2\leq2t+1$ and $(k_1,k_2)\notin\{(2t+1,2t+1),(4,3)\}$, then 
		$$|\mathcal{F}_1||\mathcal{F}_2|\leq\spn{n-t-1}{k_1-t-1}\left((t+1)\spn{n-t}{k_2-t}-t\spn{n-t-1}{k_2-t-1}\right).$$
		Moreover, equality holds if and  only if  $(\mathcal{F}_1,\mathcal{F}_2)\cong(\mathcal{C}(k_1,t),\mathcal{D}(k_2,t))$, or $k_1=k_2$ and $(\mathcal{F}_1,\mathcal{F}_2)\cong(\mathcal{D}(k_1,t),\mathcal{C}(k_2,t))$.
	\end{itemize}
\end{theorem}
It seems quite difficult to compare $|\mathcal{A}(k_1,k_2,t)||\mathcal{B}(k_2,t)|$ and $|\mathcal{C}(k_1,t)||\mathcal{D}(k_2,t)|$ in the case $k_1=k_2=2t+1$ for general $n$ and $t$. Nevertheless, one can verify that for $t\geq2$, the former quantity is larger when $n-t-1>2t^2\ln t$, whereas the latter one is larger when $2L(2t+1,t)\leq n<t+1+0.5 t^2\ln t$. The following theorem settles the problem for $r\geq3$, revealing that only one kind of optimal structure remains.
\begin{theorem}\label{thm3}
Suppose $r\geq3$, $k_1\geq k_2\geq\cdots\geq k_r\geq t+2$ and $n\geq5-k_1+2k_1\log_2k_1$. If  $\mathcal{F}_1\subseteq\spn{[n]}{k_1}$, $\mathcal{F}_2\subseteq\spn{[n]}{k_2}$, $\ldots,\;\mathcal{F}_r\subseteq\spn{[n]}{k_r}$ are non-trivial $r$-cross $t$-intersecting, then
$$\prod_{i=1}^r|\mathcal{F}_i|\leq\left((t+1)\spn{n-t}{k_r-t}-t\spn{n-t-1}{k_r-t-1}\right)\prod_{i=1}^{r-1}\spn{n-t-1}{k_i-t-1}.$$
If equality holds, then there is an $a\in[r]$ with $k_a=k_r$ and a partition $T$ consisting of $t+1$ singletons such that $\mathcal{F}_a=\mathcal{D}(k_a,T)$ and   $\mathcal{F}_i=\mathcal{C}(k_i,T)$ for each $i\in[r]\setminus\{a\}$.
\end{theorem}
We remark that the lower bound in Theorem \ref{thm3} is chosen for technical reasons, as our proof relies on applying Theorem \ref{thm1} to families satisfying intersection properties with some other parameters. From the proof in Section \ref{sectionrcross},  we see that Theorem \ref{thm3} holds for $n\geq L_0(k_1,t)$, where $L_0(k,t):=\max\{L(k,s):t+1\leq s\leq k-2\}$. One can verify that, for all $t$,  $L_0(k,t)\leq2k\log_2k-(1+o(1))k\log_2\log_2k$ as $k\to\infty$. Meanwhile,  $L_0(k_1,t)\leq2L(k_1,t)$ for all $k_1$ and $t$, and $L_0(k_1,t)\leq L(k_1,t)$ provided $t>k_1/(2\ln k_1)$.

The rest of this paper is organized as follows. In Section \ref{section2}, we first collect several useful relations on Stirling partition numbers, then establish several key lemmas, and finally prove Theorem \ref{thm1}. In Sections \ref{section3} and \ref{sectionrcross}, we utilize $t$-covers associated with the families under consideration to characterize their structure, and prove Theorems \ref{thm2} and \ref{thmwholargest}, and Theorem \ref{thm3}, respectively. In Section \ref{sectionappendix}, we prove a number of technical estimates used throughout this paper.
	\section{Erd\H{o}s-Ko-Rado type theorem}\label{section2}
	Let us collect several relations about Stirling partition numbers. A recurrence relation is 
	\begin{equation}\label{equrecurrence}
		\spn{n}{k}=\spn{n-1}{k-1}+k\spn{n-1}{k}.
	\end{equation}
	In particular, for $n\geq k\geq2$, this gives
	\begin{equation}\label{equnkn-1k}
		\spn{n}{k}>k\spn{n-1}{k}.
	\end{equation}
Let $X$ be a $t$-partition with $|\cup X|<n$, then clearly
	\begin{equation}\label{equntkt}
		\left|\left\{F\in\spn{[n]}{k}:X\subseteq F\right\}\right|=\spn{n-|\cup X|}{k-t}\leq\spn{n-t}{k-t},
	\end{equation}
with equality precisely if $|\cup X|=t$, namely, $X$ consists of singletons.

The following estimates are used frequently in the sequel.
\begin{lemma}[\cite{Lieb-1968}]\label{lemmalogconcavity}
		We have $\spn{n}{k}^2\geq\frac{k}{k-1}\spn{n}{k+1}\spn{n}{k-1}$ for $n\geq k\geq2$.
	\end{lemma}
	\begin{lemma}[\cite{Kupavskii-2023}]\label{lemmaspnntkt}
		For $m\geq r\geq 2$, we have 
	\begin{itemize}
	\item[\rm(i)]$\spn{m-1}{r}\geq\frac{1}{r}\left(2^{\frac{m-1}{r-1}}-2\right)\spn{m-1}{r-1}$, and 
	\item[\rm(ii)]$\spn{m}{r}\geq\left(2^{\frac{m-1}{r-1}}-1\right)\spn{m-1}{r-1}$.
	\end{itemize}	
	\end{lemma}
\begin{lemma}\label{lemmaspn2lkt}Let $j\geq k\geq t+2$ and $0\leq s\leq k-t-2$. The following hold.
	\begin{itemize}
		\item[\rm(i)]$\spn{n-t-s}{k-t-s}>(t+1)(j-t+1)\spn{n-t-s-1}{k-t-s-1}$ for $n\geq L(j,t)-1$.	
		\item[\rm(ii)]$\spn{n-t-s}{k-t-s}>(t+1)^2(j-t+1)^2\spn{n-t-s-1}{k-t-s-1}$ for $n\geq2L(j,t)-t-2$.
	\end{itemize}
\end{lemma}
\begin{proof} 
Let us write $a=(t+1)(j-t+1)$ and $\theta=1/(j-t-1)$ for short. For $r\geq1$ and $n\geq rL(j,t)-(r-1)(t+1)-1$, we have
\begin{align*}
	2^{\frac{n-t-s-1}{k-t-s-1}}&\geq2^{\frac{rL(j,t)-r(t+1)-1}{k-t-1}}\geq((t+1)(j-t+1))^{\frac{r(j-t+1)}{j-t-1}}/2^{\theta}\\
	&=a^{r(1+2\theta)}/2^{\theta}>a^{r+\theta}=a^r\cdot e^{\theta\ln a}>a^{r}(1+\theta\ln a)>a^r+1.
\end{align*}
It follows from Lemma \ref{lemmaspnntkt} (ii) that $\spn{n-t-s}{k-t-s}>a^r\spn{n-t-s-1}{k-t-s-1}$, and thus (i) and (ii) follow by replacing $r$ with $1$ and $2$, respectively.
\end{proof}	
Our key tool is the notion of a $t$-cover. Let $\mathcal{P}$ be a family of partitions. A \emph{$t$-cover} of $\mathcal{P}$ is a partition that shares at least $t$ blocks with every member of $\mathcal{P}$, and if such a partition exists, then the \emph{$t$-covering number} $\tau_t(\mathcal{P})$ of $\mathcal{P}$ is defined to be the minimum size of a $t$-cover. We note that a $t$-cover is not required to be a partition of the universe.

Given a pair of cross $t$-intersecting families  $\mathcal{F}\subseteq\spn{[n]}{k}$ and $\mathcal{G}\subseteq\spn{[n]}{\ell}$,  every member of $\mathcal{F}$ is a $t$-cover of $\mathcal{G}$, and vice versa. In particular, this gives  $t\leq\tau_t(\mathcal{F})\leq\ell$ and $t\leq\tau_t(\mathcal{G})\leq k$. The following lemma serves as a criterion for the triviality of cross $t$-intersecting families.
\begin{lemma}\label{lemmatt}
	Let $k,\ell\geq t$ and $n>\max\{k,\ell\}$, and let $\mathcal{F}\subseteq\spn{[n]}{k}$ and $\mathcal{G}\subseteq\spn{[n]}{\ell}$ be maximal cross $t$-intersecting families. Then $\mathcal{F}$ and $\mathcal{G}$ are trivial if and only if $\tau_t(\mathcal{F})=\tau_t(\mathcal{G})=t.$
\end{lemma}
\begin{proof}
	The lemma holds trivially for $\max\{k,\ell\}\leq t+1$. Suppose $\max\{k,\ell\}\geq t+2$. If $\mathcal{F}$ and $\mathcal{G}$ are trivial, then some $t$-partition is contained in every partition in $\mathcal{F}\cup\mathcal{G}$, and so each of the two families has $t$-covering number $t$.
	
 For the sufficiency, suppose  $\tau_t(\mathcal{F})=\tau_t(\mathcal{G})=t$, and let $X$ and $Y$ be $t$-covers of $\mathcal{F}$ and $\mathcal{G}$ with $t$ blocks, respectively. Assume to the contrary that $\mathcal{F}$ and $\mathcal{G}$ are non-trivial. Then $X$ and $Y$ are distinct, and at least one of the two families, say $\mathcal{F}$, has size larger than one.  Note that  $2\leq|\mathcal{F}|\leq\spn{n-|\cup X|}{k-t}$ yields $k\geq t+2$ and $n-|\cup X|>k-t$. We claim that $n-|\cup X|\leq\ell-t$. Indeed, assume that $n-|\cup X|\geq\ell-t+1$. Note that now the number of $\ell$-partitions of $[n]$ containing $X$ is $\spn{n-|\cup X|}{\ell-t}\geq1$. By the maximality, each of them lies in $\mathcal{G}$, and then contains $Y$ as $Y\subseteq\cap\mathcal{G}$. In particular, we obtain that $X\cup Y$ is a partition and $\spn{n-|\cup X|}{\ell-t}=\spn{n-|\cup(X\cup Y)|}{\ell-|X\cup Y|}$. From $|X\cup Y|\geq t+1$, the right hand side does not exceed $\spn{n-|\cup X|-1}{\ell-t-1}$. Then $\spn{n-|\cup X|}{\ell-t}\leq\spn{n-|\cup X|-1}{\ell-t-1}$, which is impossible from (\ref{equrecurrence}). Hence our claim holds, and then we deduce that  $\ell>k$. By symmetry, if $|\mathcal{G}|\geq2$, then $k>\ell$, which contradicts $\ell>k$. Therefore, $\mathcal{G}$ has exactly one member. Set $\mathcal{G}=\{G\}$. The average of $|\cup T|$ over $t$-subsets $T$ of $G$ equals $\binom{\ell-1}{t-1}n/\binom{\ell}{t}$, and then there are $t$ blocks of $G$ whose union has size at most  $\left\lfloor\binom{\ell-1}{t-1}n/\binom{\ell}{t}\right\rfloor=\left\lfloor tn/\ell\right\rfloor$. By the maximality again, every partition in $\spn{[n]}{k}$ containing these $t$ blocks belongs to $\mathcal{F}$. Hence $|\mathcal{F}|\geq\spn{n-\left\lfloor tn/\ell\right\rfloor}{k-t}=\spn{\left\lceil n(\ell-t)/\ell\right\rceil}{k-t}>\spn{\ell-t}{k-t}$ due to the fact that $k\geq t+2$. However, note that $|\mathcal{F}|\leq\spn{\ell-t}{k-t}$ since $X\subseteq\cap\mathcal{F}$ and $n-|\cup X|\leq\ell-t$. This contradiction proves that $\mathcal{F}$ and $\mathcal{G}$ are trivial provided $\tau_t(\mathcal{F})=\tau_t(\mathcal{G})=t.$
\end{proof}
The next two lemmas play an essential role throughout the paper in deriving bounds for families with certain intersection properties. Let $U$ be a set and let $\mathcal{F}\subset2^{U}$. For $H\subseteq U$, set
\begin{equation*}
	\mathcal{F}_H:=\{F\in\mathcal{F}:H\subseteq F\}.
\end{equation*}
	\begin{lemma}[\cite{Wen-Lv-2026}]\label{lemmainductive}Suppose $\mathcal{F}\subseteq\spn{[n]}{k}$ and $T$ is a $t$-cover of $\mathcal{F}$ with $\ell$ blocks. If $S$ is an $s$-partition with $|S\cap T|=r<t$, and $\mathcal{F}_S\neq\emptyset$, then there is a $(t+s-r)$-partition $H$ containing $S$ such that $|\mathcal{F}_S|\leq\binom{\ell-r}{t-r}|\mathcal{F}_H|$. In particular,  $|\mathcal{F}_S|\leq\binom{\ell-r}{t-r}\spn{n-s-t+r}{k-s-t+r}$.
	\end{lemma}
	\begin{lemma}\label{lemmakey}
		Suppose $k\geq t+2$, $\ell\geq t$ and  $n\geq L(\max\{k,\ell\},t)$. Let  $\mathcal{F}\subseteq\spn{[n]}{k}$ and let $\mathcal{G}$ be a collection of $t$-covers of $\mathcal{F}$, where each has at most $\ell$ blocks. Then for every $t$-partition $H$, we have 
		$$|\mathcal{F}_H|\leq\max\left\{(\ell-t+1)^{\tau_t(\mathcal{G})-t}\spn{n-\tau_t(\mathcal{G})}{k-\tau_t(\mathcal{G})},\;(\ell-t+1)^{k-t}\right\}.$$
	\end{lemma}
	\begin{proof}
		The lemma holds trivially for $\mathcal{F}_H=\emptyset$ or $\tau_t(\mathcal{G})=t$ from (\ref{equntkt}). Now suppose $\mathcal{F}_H\neq\emptyset$ and $\tau_t(\mathcal{G})>t$. Then $H$ is not a $t$-cover of $\mathcal{G}$. Hence, there exists $G_1\in\mathcal{G}$ with $\dim(H\cap G_1)<t$. Set $H_1=H$. By Lemma \ref{lemmainductive}, we can inductively choose $G_1,G_2,\ldots,G_j\in\mathcal{G}$ and partitions $H_1\subsetneqq H_2\subsetneqq\cdots\subsetneqq H_j$ such that $|H_i\cap G_i| <t$, $|H_{i+1}|=|H_{i}|+t-|H_i\cap G_i|$ and
		\begin{equation*}
			|\mathcal{F}_{H_i}|\leq\binom{|G_i|-|H_i\cap G_i|}{t-|H_i\cap G_i|}|\mathcal{F}_{H_{i+1}}|
		\end{equation*}   
		for $1\leq i\leq j-1$, and $|H_{j-1}|<\tau_t(\mathcal{G})\leq|H_j|$.
		Note that $|G_i|\leq\ell$ for $1\leq i\leq j$. Therefore,
\begin{align*}
	|\mathcal{F}_{H}|&\leq\prod_{i=1}^{j-1}\binom{\ell-|H_i\cap G_i|}{t-|H_i\cap G_i|}|\mathcal{F}_{H_j}|\leq\mathcal(\ell-t+1)^{\sum_{i=1}^{j-1}(t-|H_i\cap G_i|)}|\mathcal{F}_{H_j}|\\
	&=\mathcal(\ell-t+1)^{\sum_{i=1}^{j-1}(|H_{i+1}|-|H_i|)}|\mathcal{F}_{H_j}|=\mathcal(\ell-t+1)^{|H_j|-t}|\mathcal{F}_{H_j}|,
\end{align*}
where in the second step we use repeatedly that
\begin{equation*}\label{equltrt}
\binom{a-c}{b-c}\leq(a-b+1)^{b-c}\;\mbox{for}\; a\geq b\geq c.
\end{equation*}
The inequality above can be easily verified as  $\frac{y}{x}<\frac{y-1}{x-1}$ for $y>x>1$.

  If $|H_j|\geq k$, then $|H_j|=k$ and $H_j\in\mathcal{F}$ as $\mathcal{F}_H$ is non-empty, and thus $|\mathcal{F}_{H_j}|=1$, implying that $|\mathcal{F}_H|\leq(\ell-t+1)^{k-t}$. If $|H_j|\leq k-1$, then from (\ref{equntkt}) and Lemma \ref{lemmamono} (i), we conclude that $|\mathcal{F}_H|\leq\mathcal(\ell-t+1)^{|H_j|-t}\spn{n-|H_j|}{k-|H_j|}\leq (\ell-t+1)^{\tau_t(\mathcal{G})-t}\spn{n-\tau_t(\mathcal{G})}{k-\tau_t(\mathcal{G})}$.
\end{proof}
The lemma above suggests associating a family with a suitable set of $t$-covers, and subsequently establishing bounds related to the family considered. For simplicity, let us introduce two expressions.
\begin{align}
f(m,k,\ell,t,n)&:=\left\{\begin{array}{ll}(\ell-t+1)^{m-t}\binom{m}{t}\spn{n-m}{k-m},&{\rm if}\;t\leq m\leq k-1,\\(\ell-t+1)^{k-t}\binom{k}{t},&{\rm if}\;m=k.\\\end{array}\right.\label{equfunf}\\
g(m,k,\ell,t,n)&:=\max\{f(m,k,\ell,t,n),f(k,k,\ell,t,n)\}.\label{equfung}
\end{align}
We note that, from Lemma \ref{lemmamono} (ii) and (\ref{equfunf}), if $n\geq L(\max\{k,\ell\},t)$, then
\begin{equation*}
g(m,k,\ell,t,n)=\left\{\begin{array}{ll}f(m,k,\ell,t,n),&{\rm if}\;t\leq m\leq k-2,\\f(k,k,\ell,t,n),&{\rm if}\;k-1\leq m\leq k.\\\end{array}\right.
\end{equation*}
\begin{lemma}\label{lemmaubprod}
Suppose $k,\ell\geq t+2$ and $n\geq L(\max\{k,\ell\},t)$. If   $\mathcal{F}\subseteq\spn{[n]}{k}$ and $\mathcal{G}\subseteq\spn{[n]}{\ell}$ are cross $t$-intersecting, then
$$|\mathcal{F}||\mathcal{G}|\leq g(\tau_t(\mathcal{G}),k,\ell,t,n)g(\tau_t(\mathcal{F}),\ell,k,t,n).$$
\end{lemma}
\begin{proof}
Since $\mathcal{F}$ and $\mathcal{G}$ are cross $t$-intersecting, every member of $\mathcal{F}$ is a $t$-cover of $\mathcal{G}$, and hence $\tau_t(\mathcal{G})\leq k$. Let $T$ be a $t$-cover of $\mathcal{F}$ of size $\tau_t(\mathcal{F})$, then $\mathcal{F}=\cup_{H}\mathcal{F}_H$ with $H$ ranging over $t$-subsets of $T$, and so  $|\mathcal{F}|\leq\binom{\tau_t(\mathcal{F})}{t}|\mathcal{F}_{H_0}|$ for some $H_0\in\binom{T}{t}$. This together with Lemma \ref{lemmakey} and $\tau_t(\mathcal{G})\leq k$ yields
\begin{align*}
	\binom{\tau_t(\mathcal{G})}{t}|\mathcal{F}|&\leq\binom{\tau_t(\mathcal{F})}{t}g(\tau_t(\mathcal{G}),k,\ell,t,n).
\end{align*}
By symmetry, we have
$$\binom{\tau_t(\mathcal{F})}{t}|\mathcal{G}|\leq\binom{\tau_t(\mathcal{G})}{t} g(\tau_t(\mathcal{F}),\ell,k,t,n).$$
Multiplying the two inequalities above yields the desired inequality.
\end{proof}
\begin{proposition}\label{propell=t+1}
Let $n\geq k\geq t+2$. If $\mathcal{F}\subseteq\spn{[n]}{k}$ and $\mathcal{G}\subseteq\spn{[n]}{t+1}$ are maximal  cross $t$-intersecting families, then one of the following hold. 
\begin{itemize}
\item[\rm(i)]$\mathcal{F}=\left\{F\in\spn{[n]}{k}:|F\cap G|=t\right\}$ and $\mathcal{G}=\{G\}$ for some $G\in\spn{[n]}{t+1}$.
\item[\rm(ii)]$\mathcal{F}=\left\{F\in\spn{[n]}{k}:M\subseteq F\right\}\;\mbox{and}\;\mathcal{G}=\left\{A\cup\{\overline{\cup A}\}:A\in\binom{M}{t}\right\}$ for some partition $M$ satisfying $|M|\geq t+1$ and $M\notin\spn{[n]}{t+1}$.
\end{itemize}
Moreover, if $n\geq L(k,t)$, then $|\mathcal{F}||\mathcal{G}|\leq\spn{n-t}{k-t}$, with equality only if there is a partition $X$ consisting of $t$ singletons such that  $\mathcal{F}=\left\{F\in\spn{[n]}{k}:X\subseteq F\right\}$ and $\mathcal{G}=\{X\cup\{\overline{\cup X}\}\}$.
\end{proposition}
\begin{proof}
For all $F\in\mathcal{F}$ and $G\in\mathcal{G}$, we have $|F\cap G|=t$ as $k\geq t+2$, and hence $F\setminus G$ is a $(k-t)$-partition of the block in $G\setminus F$.
	
Suppose first that $|\mathcal{G}|=1$. Let $\mathcal{G}=\{G\}$. Then  $\mathcal{F}=\left\{F\in\spn{[n]}{k}:|F\cap G|=t\right\}$ as the families are maximal. Let $G'$ be the collection of blocks of $G$ with size at least $k-t$. It follows that $|\mathcal{F}|=\sum_{B\in G'}\spn{|B|}{k-t}$. If $s:=|G'|=1$, then we have $|\mathcal{F}|\leq\spn{n-t}{k-t}$, with equality if and only if $G$ consists of $t$ singletons and a block of size $n-t$, and $\mathcal{F}$ is the family of $k$-partitions of $[n]$ containing these  singletons. Suppose $s\geq2$. For all $B\in G'$ we have
$$n=\sum_{C\in G}|C|\geq|B|+(s-1)(k-t)+(t+1-s),$$
and so $n-t-|B|\geq(s-1)(k-t-1)$. It follows from (\ref{equnkn-1k}) that
\begin{align*}
	|\mathcal{F}|\spn{n-t}{k-t}^{-1}&<\sum_{B\in G'}(k-t)^{-(n-t-|B|)}\leq s(k-t)^{-(s-1)(k-t-1)}\leq s2^{-(s-1)}\leq1.
\end{align*}

It remains to consider the case that  $|\mathcal{G}|\geq2$. Fix an $F_0\in\mathcal{F}$. Set $\mathcal{A}=\{F_0\cap G:G\in\mathcal{G}\}$, then  $\mathcal{G}=\left\{A\cup\{\overline{\cup A}\}:A\in\mathcal{A}\right\}$. We claim that $\mathcal{F}\subseteq\left\{F\in\spn{[n]}{k}:\cup\mathcal{A}\subseteq F\right\}$.
 To prove this, let us assume to the contrary that $B_1\notin F$ for some $F\in\mathcal{F}$ and $B_1\in\cup\mathcal{A}$. Pick an  $A_1\in\mathcal{A}$ with $B_1\in A_1$, and write $G_1:=A_1\cup\{\overline{\cup A_1}\}$. Of course $F\cap G_1=G_1\setminus\{B_1\}$ and  $\overline{\cup A_1}\in F$. Since   $|\mathcal{A}|=|\mathcal{G}|\geq2$, we can pick  $A_2\in\mathcal{A}$ distinct with $A_1$. By assumption we have $|F\cap(A_2\cup\{\overline{\cup A_2}\})|=t$, and necessarily $|F\cap A_2|\geq t-1$. Note that $A_1\cup A_2\subseteq F_0$, then every block of $A_2\setminus A_1$ is a subset of $\overline{\cup A_1}$, and so it does not belong to $F$ as $\overline{\cup A_1}\in F$. Hence $F\cap A_2\subseteq A_1\cap A_2$, and consequently $F\cap A_2=A_1\cap A_2$ has size exactly $t-1$. In particular, we obtain that $\overline{\cup A_2}\in F$. However, since $A_1\cup A_2$ has $t+1<k$ blocks, it is a proper subset of $F_0$, and hence $\overline{\cup A_1}\cap\overline{\cup A_2}=\overline{\cup(A_1\cup A_2)}\neq\emptyset$. This contradicts that both $\overline{\cup A_1}$ and $\overline{\cup A_2}$ lie in $F$. Thus the claim is true. Set $M=\cup\mathcal{A}$. Then we have  $\mathcal{F}\subseteq\left\{F\in\spn{[n]}{k}:M\subseteq F\right\}$ and  $\mathcal{G}\subseteq\left\{A\cup\{\overline{\cup A}\}:A\in\binom{M}{t}\right\}$. Hence (ii) follows from the maximality of $\mathcal{F}$ and $\mathcal{G}$.  Note that $|M|\geq t+1$ and $M\notin\spn{[n]}{t+1}$ as $|\mathcal{G}|\geq2$. To complete the proof, it remains to establish the upper bound provided $n\geq L(k,t)$. Write $m=|M|$ for short. It is easy to check that $|\mathcal{F}||\mathcal{G}|=\spn{n-|\cup M|}{k-m}\binom{m}{t}$ for $t+1\leq m<k$, and $|\mathcal{F}||\mathcal{G}|=\binom{k}{t}$ for $m=k$. This together with Lemma \ref{lemmamono} (iii) yields  $|\mathcal{F}||\mathcal{G}|\leq g(m,k,t,t,n)<\spn{n-t}{k-t}$, as required.
\end{proof}
\noindent {\bf Proof of Theorem \ref{thm1}.}\; Using a standard technique (see e.g., \cite{Cao-Lu-Lv-Wang-2023}), if we can prove the theorem for $r=2$, then it also holds for all $r\geq3$. Indeed, suppose that the theorem holds for cross $t$-intersecting families, and suppose $r\geq3$ and  $\mathcal{F}_1\subseteq\spn{[n]}{k_1},\ldots,\mathcal{F}_r\subseteq\spn{[n]}{k_r}$ are $r$-cross $t$-intersecting. Then the families are pairwise cross $t$-intersecting, and  $$|\mathcal{F}_i||\mathcal{F}_j|\leq\spn{n-t}{k_i-t}\spn{n-t}{k_j-t}\;\mbox{whenever}\;i\neq j\in[r].$$ 
To see this, note that if $\max\{k_i,k_j\}\geq t+2$, then the bound holds by applying the theorem to $\mathcal{F}_i$ and $\mathcal{F}_j$. When  $k_i=k_j=t+1$, we have  $\mathcal{F}_i=\mathcal{F}_j=\{F\}$ for some $F\in\spn{[n]}{t+1}$, since every two distinct partitions in $\spn{[n]}{t+1}$ has less than $t$ blocks in common. Thus the bound is also true as $\spn{n-t}{1}=1$. It follows that 
$$\prod_{i=1}^r|\mathcal{F}_i|=\left(\prod_{i< j}|\mathcal{F}_i||\mathcal{F}_j|\right)^{\frac{1}{r-1}}\leq\left(\prod_{i< j}\spn{n-t}{k_i-t}\spn{n-t}{k_j-t}\right)^{\frac{1}{r-1}}=\prod_{i=1}^r\spn{n-t}{k_i-t}.$$
If equality holds, then necessarily $|\mathcal{F}_1||\mathcal{F}_j|=\spn{n-t}{k_1-t}\spn{n-t}{k_j-t}$ for $2\leq j\leq r$. From the characterization given in the theorem, there are $t$-partitions $X_2,\ldots, X_r$, each consisting of singletons, such that  $\mathcal{F}_1=\left\{F\in\spn{[n]}{k_1}:X_j\subseteq F\right\}$ and $\mathcal{F}_j=\left\{F\in\spn{[n]}{k_j}:X_j\subseteq F\right\}$ for $j=2,\ldots,r$. Since $|\mathcal{F}_1|=\spn{n-t}{k_1-t}$, the partitions  $X_j$ must coincide. 

Therefore, we need only to prove the theorem for $r=2$. In the case of $k_2=t+1$, the theorem holds from Proposition \ref{propell=t+1}. Suppose $k_2\geq t+2$. By Lemmas \ref{lemmaubprod} and \ref{lemmamono} (iii), we obtain that
$$|\mathcal{F}_1||\mathcal{F}_2|\leq g(\tau_t(\mathcal{F}_2),k_1,k_2,t,n)g(\tau_t(\mathcal{F}_1),k_2,k_1,t,n)\leq\spn{n-t}{k_1-t}\spn{n-t}{k_2-t}.$$
 Suppose equality holds. Then clearly $\mathcal{F}_1$ and $\mathcal{F}_2$ are maximal. By Lemma \ref{lemmamono} (iii) again, we find that $\tau_t(\mathcal{F}_1)=\tau_t(\mathcal{F}_2)=t$. It follows from Lemma \ref{lemmatt} that $\mathcal{F}_1$ and $\mathcal{F}_2$ are trivial. This together with $|\mathcal{F}_1||\mathcal{F}_2|=\spn{n-t}{k_1-t}\spn{n-t}{k_2-t}$ and (\ref{equntkt}) yields that $\mathcal{F}_i=\left\{F\in\spn{[n]}{k_i}:X\subseteq F\right\}$($i=1,2$) for some partition $X$ consisting of $t$  singletons.{\hfill$\square$}
\begin{remark}
The quantity $L(k,t)$ is chosen in order to avoid tedious calculations. An additional consideration is that the present method is unlikely to yield a significantly better bound. It might be interesting to consider whether there is an absolutely constant $C$ such that Theorem \ref{thm1} holds for $n\geq Ck_1\log(t+1)$.\end{remark}

	\section{Non-trivial cross $t$-intersecting families}\label{section3}
In this section we characterize non-trivial cross $t$-intersecting families with large product of sizes. To begin with, let us recall a useful result related to inclusion-exclusion.
\begin{lemma}\label{lemmain-ex}
	Let $A_1,A_2,\ldots,A_m$ be finite sets. Then $$|\cup_{i\in[m]}A_i|\leq{\textstyle\sum_{j\in[s]}}(-1)^{j-1}{\textstyle\sum_{J\in\binom{[m]}{j}}}|\cap_{\ell\in J}A_\ell|$$
	for odd $s$, and 
	$$|\cup_{i\in[m]}A_i|\geq{\textstyle\sum_{j\in[s]}}(-1)^{j-1}{\textstyle\sum_{J\in\binom{[m]}{j}}}|\cap_{\ell\in J}A_\ell|$$
	for even $s$.
\end{lemma}
For simplicity, we set
\begin{equation}\label{equsizeabcd}
	r_1(n,k,\ell,t):=|\mathcal{A}(k,\ell,t)||\mathcal{B}(\ell,t)|\;\mbox{and}\;r_2(n,k,\ell,t):=|\mathcal{C}(k,t)||\mathcal{D}(\ell,t)|.
\end{equation}
An easy counting argument using inclusion-exclusion yields
\begin{align}
r_1(n,k,\ell,t)&=\left(\sum_{j=1}^{\ell-t}(-1)^{j-1}\binom{\ell-t}{j}\spn{n-t-j}{k-t-j}\right)\left(\spn{n-t}{\ell-t}+t\right),\;\mbox{and}\label{equfunr1}\\
r_2(n,k,\ell,t)&=\spn{n-t-1}{k-t-1}\left((t+1)\spn{n-t}{\ell-t}-t\spn{n-t-1}{\ell-t-1}\right).\label{equfunr2}
\end{align}
Define
\begin{align}\label{equfunr}
	r(n,k,\ell,t)&:=\max\{r_1(n,k,\ell,t),r_2(n,k,\ell,t)\}.
\end{align}
A useful bound given by Lemmas \ref{lemmain-ex} and \ref{lemmaspn2lkt} (ii) is
\begin{align}\label{equlbfunr}
r(n,k,\ell,t)>\max\left\{\ell-t-\frac{1}{2(t+1)^2},\;t+1-\frac{1}{2(t+1)}\right\}\spn{n-t-1}{k-t-1}\spn{n-t}{\ell-t}
\end{align}
for $k,\ell\geq t+2$ and $n\geq2L(\max\{k,\ell\},t)$.
\begin{lemma}\label{lemmat+2small}
	Let $k,\ell\geq t+2$ and $n\geq2L(\max\{k,\ell\},t)$. Suppose that $\mathcal{F}\subseteq\spn{[n]}{k}$ and $\mathcal{G}\subseteq\spn{[n]}{\ell}$ are non-trivial cross $t$-intersecting families. 
	If $\{\tau_t(\mathcal{F}),\tau_t(\mathcal{G})\}\neq\{t,t+1\}$, then $|\mathcal{F}||\mathcal{G}|<\max\{r(n,k,\ell,t),r(n,\ell,k,t)\}$.
\end{lemma}
\begin{proof}Let us write $m_f=\tau_t(\mathcal{F})$ and $m_g=\tau_t(\mathcal{G})$ for ease of notations. From Lemma \ref{lemmatt}, we have $(m_f,m_g)\neq(t,t)$. The proof divides into two cases.\\
{\bf Case 1.}\;$m_g=t$.

In this case, we have $m_f\geq t+2$. Suppose first $\ell\geq t+3$. From Lemma \ref{lemmaubprod}, we have  $|\mathcal{F}||\mathcal{G}|\leq g(m_g,k,\ell,t,n)g(m_f,\ell,k,t,n)$. Note that Lemma \ref{lemmamono} (ii) gives   $g(m_g,k,\ell,t,n)=\spn{n-t}{k-t}$. If $\ell\geq t+4$, then by Lemmas \ref{lemmamono} (ii) and \ref{lemmaspn2lkt} (ii), we get $g(m_f,\ell,k,t,n)\leq f(t+2,\ell,k,t,n)<\spn{n-t-1}{\ell-t-1}$, and hence $|\mathcal{F}||\mathcal{G}|<\spn{n-t}{k-t}\spn{n-t-1}{\ell-t-1}<r(n,\ell,k,t)$.
If $\ell=t+3$, then $g(m_f,\ell,k,t,n)=\binom{t+3}{t}(k-t+1)^{3}$. On the other hand, note that for $n\geq2L(\max\{k,\ell\},t)$, 
\begin{align}\label{equlemmat+2small1}
	\spn{n-t-1}{2}&=2^{n-t-2}-1\geq2^t((t+1)(k-t+1))^{\ell-t+1}-1\nonumber\\&>2^tt(t+1)^3(k-t+1)^4.
\end{align}
Then it is readily seen that $g(m_f,\ell,k,t,n)<t\spn{n-t-1}{2}$. Thus we get again $|\mathcal{F}||\mathcal{G}|<t\spn{n-t}{k-t}\spn{n-t-1}{\ell-t-1}<r(n,\ell,k,t)$.

Suppose $\ell=t+2$. We claim that $|\mathcal{G}|=1$. Indeed, suppose to the contrary that $|\mathcal{G}|\geq2$, and pick $G_1\neq G_2\in\mathcal{G}$. Fix a $t$-partition $X\subseteq\cap\mathcal{G}$, and write $G_1=X\cup\{B_1,B_2\}$ and $G_2=X\cup\{B_3,B_4\}$. Since $\mathcal{F}$ and $\mathcal{G}$ are non-trivial, $\mathcal{F}\setminus\mathcal{F}_X\neq\emptyset$. Each member of $\mathcal{F}\setminus\mathcal{F}_X$ intersects both $\{B_1,B_2\}$ and $\{B_3,B_4\}$. Then necessarily $B_i\cap B_j=\emptyset$ for some  $i\in\{1,2\}$ and $j\in\{3,4\}$. Without loss of generality, suppose $B_1\cap B_3=\emptyset$. Since $\{B_1,B_2\}$ and $\{B_3,B_4\}$ are distinct partitions of $\overline{\cup X}$, we have $B_i\cap B_j\neq\emptyset$ whenever $(i,j)\in\{(1,4),(2,3),(2,4)\}$. It follows that  $F\cap\{B_1,B_2\}=\{B_1\}$, $F\cap\{B_3,B_4\}=\{B_3\}$ and $|F\cap X|=t-1$ for all $F\in\mathcal{F}\setminus\mathcal{F}_X$. Thus  $X\cup\{B_1\}$ is a $t$-cover of $\mathcal{F}$, which is impossible as $m_f\geq t+2$. This contradiction proves $|\mathcal{G}|=1$. Let $\mathcal{G}=\{G\}$. If $G$ contains $t+1$ singletons, then every member of $\mathcal{F}$ contains at least $t$ of them, which contradicts $m_f\geq t+2$. Hence $G$ contains at most $t$ singletons, and then Lemma \ref{lemmaw} yields $|\mathcal{F}||\mathcal{G}|<(t+0.6)\spn{n-t}{k-t}<r(n,\ell,k,t)$.\\
{\bf Case 2.}\;$m_g\geq t+1$.

If $m_f=t$, then $m_g\geq t+2$ due to our assumption that $\{\tau_t(\mathcal{F}),\tau_t(\mathcal{G})\}\neq\{t,t+1\}$, and so the same argument as in Case 1 applies to $\mathcal{G}$ and $\mathcal{F}$, yielding  $|\mathcal{G}||\mathcal{F}|<r(n,k,\ell,t)$. Now we suppose $m_f\geq t+1$. 

If $k,\ell\geq t+3$, then from Lemmas \ref{lemmaubprod} and \ref{lemmamono} (ii), we have 
\begin{align*}
	|\mathcal{F}||\mathcal{G}|&\leq f(t+1,k,\ell,t,n)f(t+1,\ell,k,t,n)\\
	&=(t+1)^2(k-t+1)(\ell-t+1)\spn{n-t-1}{k-t-1}\spn{n-t-1}{\ell-t-1},
\end{align*}
and hence Lemma \ref{lemmaspn2lkt} (ii) implies $|\mathcal{F}||\mathcal{G}|<\spn{n-t-1}{k-t-1}\spn{n-t}{\ell-t}<r(n,k,\ell,t)$.

 It remains, by symmetry, to consider the case of  $\ell=t+2$. Since $m_f\geq t+1$, we have $g(m_f,\ell,k,t,n)=(k-t+1)^2\binom{t+2}{t}$. By Lemma \ref{lemmamono} (ii),  $g(m_g,k,\ell,t,n)\leq f(t+1,k,\ell,t,n)=3(t+1)\spn{n-t-1}{k-t-1}$ provided $k\geq t+3$, and it equals $3^2\binom{t+2}{2}$ when $k=t+2$. Combining these with Lemma \ref{lemmaubprod}, we derive that
 $$|\mathcal{F}||\mathcal{G}|\leq9(k-t+1)^2\binom{t+2}{t}^2\spn{n-t-1}{k-t-1}.$$
 By (\ref{equlemmat+2small1}), we get $t\spn{n-t}{2}>t\spn{n-t-1}{2}>2^tt^2(t+1)^3(k-t+1)^4$. Then it is evident that $|\mathcal{F}||\mathcal{G}|<t\spn{n-t}{2}\spn{n-t-1}{k-t-1}<r(n,k,\ell,t)$.
\end{proof}
Let $\mathcal{F}$ and $\mathcal{G}$ be cross $t$-intersecting families. Denote by $\mathcal{T}(\mathcal{F})$ the set of $t$-covers of $\mathcal{F}$ with $\tau_t(\mathcal{F})$ blocks, and define $\mathcal{T}(\mathcal{G})$ in the same manner. We are in a position to consider the case of $\{\tau_t(\mathcal{F}),\tau_t(\mathcal{G})\}=\{t,t+1\}$.
\begin{lemma}\label{lemmanottintsmall}
Let $k,\ell\geq t+2$ and $n\geq2L(\max\{k,\ell\},t)$. Suppose that $\mathcal{F}\subseteq\spn{[n]}{k}$ and $\mathcal{G}\subseteq\spn{[n]}{\ell}$ are maximal  non-trivial cross $t$-intersecting families with $\tau_t(\mathcal{F})=t$ and $\tau_t(\mathcal{G})=t+1$. If  $\mathcal{T}(\mathcal{F})$ and $\mathcal{T}(\mathcal{G})$ are not cross $t$-intersecting as set systems, then $|\mathcal{F}||\mathcal{G}|<r(n,k,\ell,t)$.
\end{lemma}
\begin{proof}
By assumption, there are $X\in\mathcal{T}(\mathcal{F})$ and $T\in\mathcal{T}(\mathcal{G})$ with $|X\cap T|<t$. First, we claim that $|\cup X|\geq t+1$. To the contrary, assume that $X$ consists of singletons. Then from Lemma \ref{lemmaspn2lkt} (ii), the number of partitions $G\in\spn{[n]}{\ell}$ satisfying $X\subseteq G$ and $G\cap(T\setminus X)\neq\emptyset$ is at most 
$\sum_{B}\spn{n-|\cup X|-| B|}{\ell-t-1}\leq(t+1)\spn{n-t-1}{\ell-t-1}<\spn{n-t}{\ell-t}$, 
where $B$ ranges over $T\setminus X$. Hence there exists $G\in\spn{[n]}{\ell}$ with $X\subseteq G$ and $G\cap T=X\cap T$. Since $\mathcal{F}$ and $\mathcal{G}$ are maximal, every $\ell$-partition of $[n]$ containing $X$ belongs to $\mathcal{G}$. This leads to $|X\cap T|=|G\cap T|\geq t$ as $T$ is a $t$-cover of $\mathcal{G}$, which is a contradiction. Thus the claim is true. Similarly, we have $|\cup T|\geq t+2$.

Since $\mathcal{F}$ and $\mathcal{G}$ are non-trivial and $X\subseteq\cap\mathcal{F}$, there exists $G\in\mathcal{G}$ with $X\nsubseteq G$. Note that at least one block of $G\setminus X$ intersects $\cup(X\setminus G)$. Set $r=|X\cap G|$. Every partition in  $\mathcal{F}$ contains at least $t-r$ blocks of $G\setminus X$, and each of those blocks is disjoint from $\cup X$. Using $|\cup X|\geq t+1$, and  applying Lemma \ref{lemmamono'} with $i=1$, we obtain
\begin{equation}\label{equlemmanottintsmall1}
|\mathcal{F}|\leq\binom{\ell-r-1}{t-r}\spn{n-(t+1)-(t-r)}{k-t-(t-r)}\leq(\ell-t)\spn{n-t-2}{k-t-1}.
\end{equation}

Note that $T$ is a $t$-cover of $\mathcal{G}$, and at least one block in it has size larger than one, then we derive $$|\mathcal{G}|\leq\sum_{B\in T}\spn{n-|\cup(T\setminus\{B\})|}{\ell-t}\leq t\spn{n-t-1}{\ell-t}+\spn{n-t}{\ell-t}<\frac{\ell}{\ell-t}\spn{n-t}{\ell-t},$$ where the last inequality follows from (\ref{equnkn-1k}). If $k\geq t+3$, then by (\ref{equnkn-1k}) and (\ref{equlbfunr}),
$$|\mathcal{F}||\mathcal{G}|<\ell\spn{n-t-2}{k-t-1}\spn{n-t}{\ell-t}<\frac{\ell}{2}\spn{n-t-1}{k-t-1}\spn{n-t}{\ell-t}<r(n,k,\ell,t).$$
Suppose $k=t+2$. Now (\ref{equlemmanottintsmall1}) reduces to $|\mathcal{F}|\leq\ell-t$. When $|\mathcal{F}|=1$, we get $|\mathcal{F}||\mathcal{G}|<\frac{\ell}{\ell-t}\spn{n-t}{\ell-t}<r(n,k,\ell,t)$. It remains only to consider the case that  $|\mathcal{F}|\geq2$. Pick $F_1\neq F_2\in\mathcal{F}$. By the same token as in  Case 1 of the proof of Lemma \ref{lemmat+2small}, we obtain that $|G\cap X|=t-1$ whenever  $G\in\mathcal{G}\setminus\mathcal{G}_X$, and there are $B_i\in F_i\setminus X$($i=1,2$) such that $\{B_1,B_2\}\subseteq G$ for all $G\in\mathcal{G}\setminus\mathcal{G}_X$. Hence $|\mathcal{G}\setminus\mathcal{G}_X|\leq t\spn{n-t-1}{\ell-t-1}$, and finally this together with (\ref{equnkn-1k}) and Lemma \ref{lemmaspn2lkt} (ii) gives 
\begin{align*}
|\mathcal{F}||\mathcal{G}|&<(\ell-t)\left(\spn{n-|\cup X|}{\ell-t}+t\spn{n-t-1}{\ell-t-1}\right)\\
&<(\ell-t)\left(\frac{1}{\ell-t}+\frac{1}{(t+1)(\ell-t+1)^2}\right)\spn{n-t}{\ell-t}<r(n,k,\ell,t),
\end{align*}
as desired.
\end{proof}
\begin{lemma}\label{lemmathm2}
Let $k,\ell\geq t+2$ and $n\geq2L(\max\{k,\ell\},t)$. Let $\mathcal{F}\subseteq\spn{[n]}{k}$ and $\mathcal{G}\subseteq\spn{[n]}{\ell}$ be maximal non-trivial $t$-intersecting with $\tau_t(\mathcal{F})=t$ and $\tau_t(\mathcal{G})=t+1$, and let $\mathcal{T}(\mathcal{F})$ and $\mathcal{T}(\mathcal{G})$ be cross $t$-intersecting as set systems. If $|\mathcal{F}||\mathcal{G}|\geq\max\{ r(n,k,\ell,t),r(n,\ell,k,t)\}$, then one of the following hold. 
\begin{itemize}
	\item[\rm(i)]$(\mathcal{F},\mathcal{G})\cong(\mathcal{A}(k,\ell,t),\mathcal{B}(\ell,t))$.
	\item[\rm(ii)]$(\mathcal{F},\mathcal{G})\cong(\mathcal{C}(k,t),\mathcal{D}(\ell,t))$.
	\item[\rm(iii)]$(k,\ell)=(3,3)$, and $\mathcal{F}=\mathcal{A}(3,\{A\},M)$ and $\mathcal{G}=\mathcal{B}(3,\{A\},M)$, where $A$ is a singleton and $M$ is a $3$-partition containing $A$ with $3<|\cup M|<n$.
\end{itemize}
\end{lemma}
\begin{proof}Let us prove that, if the pair  $(\mathcal{F},\mathcal{G})$ conforms to none of the structures listed in (i), (ii) and (iii), then $|\mathcal{F}||\mathcal{G}|<\max\{ r(n,k,\ell,t),r(n,\ell,k,t)\}$. Note that  $X\subseteq T$ for all $X\in\mathcal{T}(\mathcal{F})$ and $T\in\mathcal{T}(\mathcal{G})$. 
	
First, if $|\mathcal{T}(\mathcal{F})|\geq2$, then $\mathcal{T}(\mathcal{G})=\{\cup(\mathcal{T}(\mathcal{F}))\}$. Write $T=\cup(\mathcal{T}(\mathcal{F}))$. By the maximality, we have  $\mathcal{F}=\mathcal{C}(k,T)$ and $\mathcal{G}=\mathcal{D}(\ell,T)$. If $|\cup T|=t+1$, then $(\mathcal{F},\mathcal{G})\cong(\mathcal{C}(k,t),\mathcal{D}(\ell,t))$. When $|\cup T|\geq t+2$ , we have from (\ref{equnkn-1k}) that $|\mathcal{G}|\leq t\spn{n-t-1}{\ell-t}+\spn{n-t}{\ell-t}<\frac{\ell}{\ell-t}\spn{n-t}{\ell-t}$, and hence (\ref{equlbfunr}) yields  $|\mathcal{F}||\mathcal{G}|<\frac{\ell}{\ell-t}\spn{n-t-1}{k-t-1}\spn{n-t}{\ell-t}<r(n,k,\ell,t).$

In what follows, we suppose $|\mathcal{T}(\mathcal{F})|=1$. Set $\mathcal{T}(\mathcal{F})=\{X\}$, and set $T=X\cup\{B_T\}$ for each  $T\in\mathcal{T}(\mathcal{G})$. Then $M:=\cup(\mathcal{T}(\mathcal{G}))=X\cup\{B_T:T\in\mathcal{T}(\mathcal{G})\}$. Since $\mathcal{F}$ and $\mathcal{G}$ are non-trivial and $X\subseteq\cap\mathcal{F}$, we have $\mathcal{G}\setminus\mathcal{G}_X\neq\emptyset$. Let  $G\in\mathcal{G}\setminus\mathcal{G}_X$. It follows that $|G\cap X|=t-1$ and $B_T\in G$ whenever $T\in\mathcal{T}(\mathcal{G})$. Then  there exists $B\in X$ with $G\cap M=M\setminus\{B\}$, and therefore,
\begin{equation}\label{equlemmathm23}
\mathcal{G}\setminus\mathcal{G}_X\subseteq\cup_{B\in X}\mathcal{G}_{M\setminus\{B\}}.
\end{equation}
In particular, we obtain that $M$ is a partition, and $M\setminus\{B\}\subsetneqq G$, which implies that $m:=|M|\leq\ell$.
 
We proceed by dealing with the case when  $m=\ell$. Now 
\begin{equation}\label{equlemmathm21}
\mathcal{G}\setminus\mathcal{G}_X\subseteq\{(M\setminus\{B\})\cup\{B\cup\overline{\cup M}\}:B\in X\},	
\end{equation} and hence $\mathcal{G}\subseteq\mathcal{B}(\ell,X,M)$. Since  $\mathcal{G}\setminus\mathcal{G}_X$ is non-empty, every member of $\mathcal{F}$ intersects $M\setminus X$, and so $\mathcal{F}\subseteq\mathcal{A}(k,X,M)$. It follows from the maximality that  $$\mathcal{F}=\mathcal{A}(k,X,M)\;\mbox{and}\;\mathcal{G}=\mathcal{B}(\ell,X,M).$$
If $|\cup M|=\ell$, namely, $M$ consists of singletons, then $(\mathcal{F},\mathcal{G})\cong(\mathcal{A}(k,\ell,t),\mathcal{B}(\ell,t))$. Suppose $|\cup M|\geq\ell+1$. Note that
 $|\mathcal{G}|=\spn{n-|\cup X|}{\ell-t}+t\leq|\mathcal{B}(\ell,t)|$, with equality precisely if $|\cup X|=t$. When $k\geq t+3$, we have 
\begin{align*}
|\mathcal{F}|&<\sum_{B\in M\setminus X}\spn{n-|(\cup X)\cup B|}{k-t-1}\leq(\ell-t-1)\spn{n-t-2}{k-t-1}+\spn{n-t-1}{k-t-1}.
\end{align*}
Combining this with Lemma \ref{lemmain-ex} (ii), (\ref{equnkn-1k})  and Lemma \ref{lemmaspn2lkt} (ii), we obtain that
\begin{align*}
|\mathcal{A}(k,\ell,t)|-|\mathcal{F}|&>(\ell-t-1)\left(\spn{n-t-1}{k-t-1}-\spn{n-t-2}{k-t-1}\right)-\binom{\ell-t}{2}\spn{n-t-2}{k-t-2}\\
&>\frac{\ell-t-1}{2}\spn{n-t-1}{k-t-1}-\binom{\ell-t}{2}\spn{n-t-2}{k-t-2}>0.
\end{align*}
Recall from (\ref{equsizeabcd}) that $r_1(n,k,\ell,t)$ stands for $|\mathcal{A}(k,\ell,t)||\mathcal{B}(\ell,t)|$. Thus  $|\mathcal{F}||\mathcal{G}|<r_1(n,k,\ell,t)\leq r(n,k,\ell,t)$. When $k=t+2$, we have $|\mathcal{F}|=\ell-t=|\mathcal{A}(k,\ell,t)|$. It follows that $|\mathcal{F}||\mathcal{G}|<r_1(n,k,\ell,t)\leq r(n,k,\ell,t)$ for $|\cup X|\geq t+1$. In the case of $k=t+2$ and $|\cup X|=t$, the product $|\mathcal{F}||\mathcal{G}|$ equals $r_1(n,t+2,\ell,t)$. In particular, for $(k,\ell)=(3,3)$, we have $r_1(n,3,3,1)>r_2(n,3,3,1)$, and $\mathcal{F}$ and $\mathcal{G}$ conform to the structure in (iii) by setting $\{A\}=X$. To deduce $|\cup M|<n$, just using (\ref{equlemmathm21}) and the fact that  $\mathcal{G}\setminus\mathcal{G}_X\neq\emptyset$. For $(k,\ell)\neq(3,3)$, we obtain from Lemma \ref{lemmar1r2} (ii) that $|\mathcal{F}||\mathcal{G}|<r_2(n,t+2,t+2,t)$ for $\ell=t+2\geq4$, and from Lemma \ref{lemmar1kllk} that $|\mathcal{F}||\mathcal{G}|<r_1(n,\ell,t+2,t)$ for $\ell>t+2$, where the restriction $n\geq(t+1)+(k-t)(\ell-t)$ is guaranteed as $k=t+2$ and $n\geq2L(\ell,t)>2\ell-t$.

It remains to prove  $|\mathcal{F}||\mathcal{G}|<r(n,k,\ell,t)$ provided $m\leq\ell-1$. Fix a $G_0\in\mathcal{G}\setminus\mathcal{G}_X$, and set $G_0\cap M=M\setminus\{B\}$ for $B\in X$. Then $\mathcal{F}\setminus(\cup_{T\in\mathcal{T}(\mathcal{G})}\mathcal{F}_{T})=\cup_{C\in H}\mathcal{F}_{X\cup\{C\}}$, where $H$ denotes the set of blocks of $G_0\setminus M$ which are disjoint from $B$. Since $B$ does not lie in $G_0$, at least one block of $G_0$ intersects $B$, which implies that $|H|\leq|G_0\setminus M|-1=\ell-m$. Let $C\in H$, then $X\cup\{C\}\notin\mathcal{T}(\mathcal{G})$ as $H\cap M=\emptyset$. Hence there exists $G\in\mathcal{G}\setminus\mathcal{G}_X$ with $|(X\cup\{C\})\cap G|=t-1$, and so Lemma \ref{lemmainductive} yields $|\mathcal{F}_{X\cup\{C\}}|\leq(\ell-t+1)\spn{n-t-2}{k-t-2}$. Therefore, we obtain $$|\mathcal{F}|\leq(m-t)\spn{n-t-1}{k-t-1}+(\ell-m)(\ell-t+1)\spn{n-t-2}{k-t-2}.$$
A straightforward comparison using Lemma \ref{lemmaspn2lkt} (ii) shows that the right-hand side is maximized when $m=\ell-1$, and hence
\begin{equation}\label{equlemmathm2sizef}
|\mathcal{F}|\leq\left\{\begin{array}{ll}(\ell-t-1)\spn{n-t-1}{k-t-1}+(\ell-t+1)\spn{n-t-2}{k-t-2},&{\rm if}\;m\geq t+2,\\\spn{n-t-1}{k-t-1}+(\ell-t-1)(\ell-t+1)\spn{n-t-2}{k-t-2},&{\rm if}\;m=t+1.\\\end{array}\right.
\end{equation}
Next, we have to bound the size of  $\mathcal{G}\setminus\mathcal{G}_X$. Fix an $F_0\in\mathcal{F}$. By the same token we have 
$\mathcal{G}\setminus\mathcal{G}_X=\cup_{B,C}\mathcal{G}_{(X\setminus\{B\})\cup\{C\}}$, where $B$ and $C$  range over $X$ and $F_0\setminus X$, respectively. In addition, $(X\setminus\{B\})\cup\{C\}$ does not lie in $\mathcal{T}(\mathcal{F})$ for every choice of $B$ and $C$. By combining these with Lemma \ref{lemmainductive}, we obtain that $|\mathcal{G}\setminus\mathcal{G}_X|\leq t(k-t)(k-t+1)\spn{n-t-1}{\ell-t-1}$. Observe that when $m\geq t+2$, a better bound given by (\ref{equlemmathm23}) is $|\mathcal{G}\setminus\mathcal{G}_X|\leq\cup_{B\in X}|\mathcal{G}_{M\setminus\{B\}}|\leq t\spn{n-t-1}{\ell-t-1}$. Finally, we deduce from (\ref{equlemmathm2sizef}) and Lemma \ref{lemmaspn2lkt} (ii) that 
$$\frac{|\mathcal{F}||\mathcal{G}|}{\spn{n-t-1}{k-t-1}\spn{n-t}{\ell-t}}<
\left\{\begin{array}{ll}\left(\ell-t-1+\frac{1}{(t+1)^2(\ell-t+1)}\right)\left(1+\frac{1}{(t+1)(\ell-t+1)^2}\right),&{\rm if}\;m\geq t+2,\\
	\left(1+\frac{1}{(t+1)^2}\right)\left(1+\frac{1}{t+1}\right),&{\rm if}\;m=t+1.\\
\end{array}\right.$$
Then it is routine to check that  $|\mathcal{F}||\mathcal{G}|\spn{n-t-1}{k-t-1}^{-1}\spn{n-t}{\ell-t}^{-1}$ is less then $\ell-t-\frac{1}{2(t+1)^2}$, and thus $|\mathcal{F}||\mathcal{G}|<r(n,k,\ell,t)$. This completes the proof.
\end{proof}
\noindent{\bf Proof of Theorem \ref{thm2}.}\;This follows from Lemmas \ref{lemmat+2small}, \ref{lemmanottintsmall} and  \ref{lemmathm2}.{\hfill$\square$}

\noindent{\bf Proof of Theorem  \ref{thmwholargest}.}\;Let us recall the fact (\ref{equsizeabcd}) that $r_1(n,k,\ell,t)$ and $r_2(n,k,\ell,t)$ are given by the quantities $|\mathcal{A}(k,\ell,t)||\mathcal{B}(\ell,t)|$ and $|\mathcal{C}(k,t)||\mathcal{D}(\ell,t)|$, respectively. By Theorem \ref{thm2}, the proof reduces to determining the largest among $r_1(n,k,\ell,t)$, $r_1(n,\ell,k,t)$, $r_2(n,k,\ell,t)$ and $r_2(n,\ell,k,t)$. It is finished by Lemmas \ref{lemmar2kllk}, \ref{lemmar1r2} and \ref{lemmar1kllk}. {\hfill$\square$}
\begin{proposition}\label{propkl=33}
Let $n\geq2L(3,1)$ and   $\mathcal{F},\mathcal{G}\subseteq\spn{[n]}{3}$ be non-trivial cross $1$-intersecting families maximizing $|\mathcal{F}||\mathcal{G}|$. If   $|\mathcal{F}|\leq|\mathcal{G}|$, then  $\mathcal{F}=\mathcal{A}(3,\{A\},M)$ and $\mathcal{G}=\mathcal{B}(3,\{A\},M)$ for a singleton $A$ and a $3$-partition $M$ containing $A$ with $|\cup M|<n$.
\end{proposition}
\begin{proof}
From Lemmas \ref{lemmat+2small}, \ref{lemmanottintsmall} and \ref{lemmathm2}, we obtain that either $\mathcal{F}=\mathcal{A}(3,\{A\},M)$ and $\mathcal{G}=\mathcal{B}(3,\{A\},M)$ for a singleton $A$ and a $3$-partition $M$ with $|\cup M|<n$ and $A\in M$, or $(\mathcal{F},\mathcal{G})\cong(\mathcal{C}(3,1),\mathcal{D}(3,1))$. Families with these two structures have respective products of sizes $r_1(n,3,3,1)=2\left(\spn{n-1}{2}+1\right)$ and $r_2(n,3,3,1)=2\spn{n-1}{2}-1<r_1(n,3,3,1)$. This finishes the proof.
\end{proof}
\section{Non-trivial $r$-cross $t$-intersecting families}\label{sectionrcross}
In this section, we investigate non-trivial $r$-cross $t$-intersecting families for $r\geq3$. First, the case in which some uniformity equals $t+1$ is quite clear.
\begin{proposition}\label{propell=t+1'}
	Let $r\geq3$, $n\geq k_1\geq k_2\geq\cdots\geq k_r=t+1$ and $k_1\geq t+2$. If $\mathcal{F}_1\subseteq\spn{[n]}{k_1}$, $\mathcal{F}_2\subseteq\spn{[n]}{k_2},\ldots$,$\mathcal{F}_r\subseteq\spn{[n]}{k_r}$ are maximal non-trivial $r$-cross $t$-intersecting families, then one of the following hold.
	\begin{itemize}
		\item[\rm(i)]$k_2=t+1$, and there exists $G\in\spn{[n]}{t+1}$ such that  $\mathcal{F}_1=\left\{F\in\spn{[n]}{k_1}:|F\cap G|=t\right\}$ and $\mathcal{F}_2=\cdots=\mathcal{F}_r=\{G\}$.
		\item[\rm(ii)]$k_{r-1}\geq t+2$, and there is a partition $M$  satisfying $t+1\leq|M|\leq k_{r-1}$ and $M\notin\spn{[n]}{t+1}$ such that  $\mathcal{F}_i=\left\{F\in\spn{[n]}{k_i}:M\subseteq F\right\}$ for $1\leq i\leq r-1$ and $\mathcal{F}_r=\left\{A\cup\{\overline{\cup A}\}:A\in\binom{M}{t}\right\}$.
	\end{itemize}
\end{proposition}
\begin{proof}
Suppose first that $|\mathcal{F}_r|=1$. Let $\mathcal{F}_r=\{F_r\}$. Fix an $F_1\in\mathcal{F}_1$, then $F_1$ and $F_r$ have exactly $t$ blocks in common as they are partitions of $[n]$ and $k_1\geq t+2$. Since $\mathcal{F}_1,\ldots,\mathcal{F}_r$ are $r$-cross $t$-intersecting, $F_1\cap F_r$ is contained in each member of $\cup_{2\leq i\leq r-1}\mathcal{F}_i$. Note that the families are non-trivial, then there exists $F_1'\in\mathcal{F}_1$ with $F_1\cap F_r\nsubseteq F_1'$, and so $F_r=(F_1\cap F_r)\cup(F_1'\cap F_r)$. On the other hand, we also have $F_1'\cap F_r\subseteq F$ whenever $F\in\cup_{2\leq i\leq r-1}\mathcal{F}_i$, and then $F_r\subseteq\cup_{2\leq i\leq r-1}\mathcal{F}_i$. Thus $\mathcal{F}_2=\cdots=\mathcal{F}_{r-1}=\{F_r\}$. Hence $\mathcal{F}_1$ consists precisely of those partitions that are  $t$-intersect with $F_r$ as the families are maximal. Therefore, $\mathcal{F}_1,\ldots,\mathcal{F}_r$ conform to the structure in (i) by setting $G=F_r$. 

Suppose $|\mathcal{F}_r|\geq2$. The same argument as in Proposition \ref{propell=t+1} gives  
$$\mathcal{F}_1\subseteq\left\{F\in\spn{[n]}{k_1}:\cup\mathcal{A}\subseteq F\right\}\;\mbox{and}\;\mathcal{F}_r=\left\{A\cup\{\overline{\cup A}\}:A\in\mathcal{A}\right\},$$
where $\mathcal{A}$ is a family of $t$-partitions with $|\mathcal{A}|\geq2$. For all $A\in\mathcal{A}$ and $F_1\in\mathcal{F}_1$, we have $F_1\cap (A\cup\{\overline{\cup A}\})=A$, and then $A\subseteq F$ whenever $F\in\mathcal{F}_2\cup\cdots\cup\mathcal{F}_{r-1}$ due to the fact that  $\mathcal{F}_1,\ldots,\mathcal{F}_r$ are $r$-cross $t$-intersecting. Hence $\cup\mathcal{A}\subseteq\cap\mathcal{F}_i$ for $2\leq i\leq r-1$. Then the maximality gives $\mathcal{F}_i=\left\{F\in\spn{[n]}{k_i}:\cup\mathcal{A}\subseteq F\right\}$ for $1\leq i\leq r-1$, and then gives $\mathcal{A}=\binom{\cup\mathcal{A}}{t}$. By setting $M=\cup\mathcal{A}$, we see that the families conform to the structure in (ii). Since $|\mathcal{F}_r|\geq2$ and $M\subseteq\cap\mathcal{F}_{r-1}$, we have  $t+1\leq|M|\leq k_{r-1}$ and $M\notin\spn{[n]}{t+1}$, which implies that $k_{r-1}\geq t+2$.
\end{proof}
In what follows, we may assume that every uniformity is at least $t+2$. Let us introduce an important ingredient. To be precise, suppose $r\geq3$, and  $\mathcal{F}_1\subseteq\spn{[n]}{k_1}$, $\mathcal{F}_2\subseteq\spn{[n]}{k_2},\ldots$, $\mathcal{F}_r\subseteq\spn{[n]}{k_r}$ are $r$-cross $t$-intersecting. Define
\begin{equation*}
\mathcal{G}_i:=\left\{\cap_{j\neq i}F_j:F_j\in\mathcal{F}_j,\;j\in[r]\setminus\{i\}\right\},\;i=1,2,\ldots,r.
\end{equation*}
It turns out that the families $\mathcal{G}_1,\ldots, \mathcal{G}_r$ accurately capture the property `$r$-cross', rather than merely describing that $\mathcal{F}_1,\ldots,\mathcal{F}_r$ are pairwise cross $t$-intersecting. Indeed, note that for each $i\in[r]$, the families  $\mathcal{F}_i$ and $\mathcal{G}_i$ are cross $t$-intersecting, and thus  every member of $\mathcal{G}_i$ is a $t$-cover of $\mathcal{F}_i$, and vice versa. In addition, let us write $$s_i:=\min\{|G|:G\in\mathcal{G}_i\},\;i=1,2,\ldots,r.$$ 
We find immediately that for each $a\in[r]$, the quantity $s_a\geq t$, and the families  $\{\mathcal{F}_i\}_{i\in[r]\setminus\{a\}}$ are $(r-1)$-cross $s_a$-intersecting but not $(r-1)$-cross $(s_a+1)$-intersecting. This naturally leads to the strategy of considering whether some of the families satisfy stronger intersection properties. Precisely, the proof of Theorem \ref{thm3} proceeds by considering two cases, according to whether there exists some index $a$ such that $\{\mathcal{F}_i\}_{i\in[r]\setminus\{a\}}$ are $(r-1)$-cross $(t+1)$-intersecting.

Before proving the theorem, let us write 
\begin{equation}\label{equfunh}
	h(m,k,t,n):=\left|\left\{F\in\spn{[n]}{k}:|F\cap[[m]]|\geq t\right\}\right|
\end{equation}
for simplicity. In particular, we have $h(t+1,k,t,n)=(t+1)\spn{n-t}{k-t}-t\spn{n-t-1}{k-t-1}$, which equals the size of $\mathcal{D}(k,t)$. We also define
\begin{equation}\label{equfunphi}
	\varphi(m,a;k_1,\ldots,k_r,n):=h(m,k_a,t,n)\prod_{i\neq a}\spn{n-m}{k_i-m}.
\end{equation}
When there is no likelihood of 
confusion, we simply write $\varphi(m,a)$. The next fact required is
\begin{equation}\label{equublkt}
5-k+2k\log_2k>L(k,t)\;\mbox{for}\;k\geq t+2.
\end{equation}
To see this, just note that the AM-GM inequality yields $L(k,t)\leq k-1+2k\log_2\left(\frac{k+2}{2}\right)$, and then (\ref{equublkt}) follows from $\log_2(k+2)-\log_2k<2/(k\ln 2)$.

\begin{lemma}\label{lemmas_a>t}
Suppose $r\geq3$, $k_1\geq k_2\geq\cdots\geq k_r\geq t+2$ and $n\geq 5-k_1+2k_1\log_2k_1$. Let   $\mathcal{F}_1\subseteq\spn{[n]}{k_1}$, $\mathcal{F}_2\subseteq\spn{[n]}{k_2}$, $\ldots,\;\mathcal{F}_r\subseteq\spn{[n]}{k_r}$ be non-trivial $r$-cross $t$-intersecting. If $\max\{s_i:i\in[r]\}>t$, then  $\prod_{i=1}^r|\mathcal{F}_i|\leq\varphi(t+1,r)$. Moreover, if equality holds, then there is an $a\in[r]$ with $k_a=k_r$ and a partition $T$ consisting of $t+1$ singletons such that $\mathcal{F}_a=\mathcal{D}(k_a,T)$ and   $\mathcal{F}_i=\mathcal{C}(k_i,T)$ for each $i\in[r]\setminus\{a\}$.
\end{lemma}
\begin{proof}
Fix an index $a\in[r]$ with $s_a\geq t+1$, and pick $G\in\mathcal{G}_a$ with $|G|=s_a$. Then $\mathcal{F}_a\subseteq\left\{F\in\spn{[n]}{k_a}:|G\cap F|\geq t\right\}=:\mathcal{W}$. Note that $s_a\leq\max\{k_i:i\in[r]\setminus\{a\}\}$ due to the definition of $\mathcal{G}_a$.

Suppose first that   $s_a=\min\{k_i:i\in[r]\setminus\{a\}\}$. Now $s_a\geq t+2$. Since $\{\mathcal{F}_i\}_{i\in[r]\setminus\{a\}}$ are $(r-1)$-cross $s_a$-intersecting, we have $k_i=s_a$ and $\mathcal{F}_i=\{F\}$ whenever $i\in[r]\setminus\{a\}$, where $F$ is an $s_a$-partition of $[n]$. Then $F=G$ by the definition of $\mathcal{G}_a$, and then $\mathcal{F}_a=\mathcal{W}$  by the maximality of $\mathcal{F}_1,\ldots,\mathcal{F}_r$. If $s_a\geq t+3$, then 
$$|\mathcal{F}_a|\leq\binom{s_a}{t}\spn{n-t}{k_a-t}
<h(t+1,k_a,t,n)\spn{n-t-1}{s_a-t-1}^2\leq\varphi(t+1,a).$$
To see the second inequality, note that $h(t+1,k_a,t,n)>\spn{n-t}{k_a-t}$, and Lemma \ref{lemmamono''} gives $\spn{n-t-1}{s_a-t-1}^2>\binom{s_a}{t}^2/(t+1)^2\geq\binom{s_a}{t}$. If  $s_a=t+2$, then from Lemma \ref{lemmaw}, we have $|\mathcal{F}_a|\leq h(t+1,k_a,t,n)=\varphi(t+1,a)$. Hence we derive from Lemma \ref{lemmafunhmono} that $\prod_{i=1}^r|\mathcal{F}_i|=|\mathcal{F}_a|\leq\varphi(t+1,a)\leq\varphi(t+1,r)$. Suppose the product equals  $\varphi(t+1,r)$. Then $s_a=t+2$, and we have $k_a=k_r$ and $G$ contains exactly $t+1$ singletons from Lemmas \ref{lemmafunhmono}, \ref{lemmaw} and \ref{lemmaspn2lkt} (i). In this case, let $T$ denote the set of these singletons, then $G=T\cup\{\overline{\cup T}\}$, and hence $\mathcal{F}_a=\mathcal{D}(k_a,T)$, and $\mathcal{F}_i=\{G\}=\mathcal{C}(k_i,T)$ for all $i\neq a$. 

Suppose  $s_a<\min\{k_i:i\in[r]\setminus\{a\}\}$. If $\max\{k_i:i\in[r]\setminus\{a\}\}=s_a+1$, then clearly $\prod_{i\neq a}|\mathcal{F}_i|=1$. Suppose $\max\{k_i:i\in[r]\setminus\{a\}\}\geq s_a+2$. Note that  $L(x,y)$ is certainly  increasing on $x$. Then from (\ref{equublkt}), we have $n\geq L(\max\{k_i:i\in[r]\setminus\{a\}\},s_a)$.  Hence Theorem \ref{thm1} applies to $\{\mathcal{F}_i\}_{i\in[r]\setminus\{a\}}$, and therefore 
\begin{equation*}\label{equthm3case11}
	\prod_{i\neq a}|\mathcal{F}_i|\leq\prod_{i\neq a}\spn{n-s_a}{k_i-s_a}.
\end{equation*}
It remains to bound $|\mathcal{F}_a|$. By assumption $k_i>s_a$ for each $i\in[r]\setminus\{a\}$, and so $G$ is not a partition of $[n]$ as by definition $G$ is a proper subset of some member of $\mathcal{F}_i$. When $s_a=t+1$, we have $\mathcal{F}_a\subseteq\left\{F\in\spn{[n]}{k_a}:|F\cap P|\geq t\right\}$ for $P:=G\cup\{\overline{\cup G}\}\in\spn{[n]}{t+2}$. Then Lemma \ref{lemmaw} (i) yields $|\mathcal{F}_a|\leq h(t+1,k_a,t,n)$. If  $s_a\geq t+2$, then Lemma \ref{lemmaw} (ii) gives  $|\mathcal{F}_a|\leq6h(s_a,k_a,t,n)$. Note also that Lemma \ref{lemmaspn2lkt} (i) gives that the quotient $\prod_{i\neq a}\spn{n-t-1}{k_i-t-1}/\spn{n-s_a}{k_i-s_a}$ exceeds $((t+1)(k_1-t+1))^{(r-1)(s_a-t-1)}>6$ for $s_a\geq t+2$. Thus we deduce that
\begin{align*}
	\prod_{i=1}^r|\mathcal{F}_i|\leq h(s_a,k_a,t,n)\prod_{i\neq a}\spn{n-s_a}{k_i-s_a}= \varphi(s_a,a)\leq\varphi(t+1,r),
\end{align*}
where the last inequality follows from  Lemma \ref{lemmafunhmono}. Suppose equality holds. Then $s_a=t+1$ and $k_a=k_r$, and so $\mathcal{F}_a=\mathcal{W}=\mathcal{D}(k_a,G)$. In addition, by Theorem \ref{thm1}, there is a $(t+1)$-partition $T$ consisting of singletons  such that $\mathcal{F}_i=\left\{F\in\spn{[n]}{k_i}:T\subseteq F\right\}=\mathcal{C}(k_i,T)$ whenever $i\neq a$. Since $G\in\mathcal{G}_a$ and $T\subseteq\cap_{i\neq a}(\cap\mathcal{F}_i)$, we have $T\subseteq G$, and so $T=G$ as $|G|=s_a=t+1$.
\end{proof}

To prove Theorem \ref{thm3}, the case left over by the lemma above is that $s_1=\cdots=s_r=t$, namely, every $r-1$ among the families are not $(r-1)$-cross $(t+1)$-intersecting.
\begin{lemma}\label{lemmathm33}
Suppose $r\geq3$ and $n>k_1\geq k_2\geq\cdots\geq k_r\geq t+2$. Let  $\mathcal{F}_1\subseteq\spn{[n]}{k_1}$, $\mathcal{F}_2\subseteq\spn{[n]}{k_2}$, $\ldots,\;\mathcal{F}_r\subseteq\spn{[n]}{k_r}$ be non-trivial $r$-cross $t$-intersecting. If $s_i=t$ for all $i\in[r]$, then the following hold.
\begin{itemize}
\item[\rm (i)]For each $j\in[r]$, every $t$-cover of $\mathcal{G}_j$ is a $(t+1)$-cover of all but at most one of $\mathcal{F}_i$ with $i$ ranging over $[r]\setminus\{j\}$.
\item[\rm(ii)]We have $\tau_t(\mathcal{F}_i)=t$ and $\tau_t(\mathcal{G}_i)\geq t+2$ for all $i\in[r]$.
\end{itemize}
\end{lemma}
\begin{proof}
Let $j\in[r]$ and $T_j$ be a $t$-cover of $\mathcal{G}_j$, then by the definition of $\mathcal{G}_j$,
\begin{equation}\label{equthm3case21}
	|T_j\cap(\cap_{i\neq j}F_i)|\geq t\;\mbox{for all}\;F_i\in\mathcal{F}_i,\;i\in[r]\setminus\{j\}.
\end{equation}
We claim that  $\tau_t(\mathcal{G}_j)\geq t+1$ for $1\leq j\leq r$. Indeed, assume to the contrary that $\tau_t(\mathcal{G}_j)=t$ for some $j\in[r]$, and pick a $t$-cover $T_j$ of $\mathcal{G}_j$ with $t$ blocks. Clearly (\ref{equthm3case21}) implies $T_j\subseteq\cap\mathcal{F}_i$ for all $i\neq j$. Since $s_j=t$, there exist  $F_i\in\mathcal{F}_i,i\in[r]\setminus\{j\}$ such that $G_j:=\cap_{i\neq j}F_i$ has exactly $t$ blocks, and then $T_j=G_j$. However, note that $G_j\subseteq\cap\mathcal{F}_j$ as $\mathcal{F}_1,\ldots,\mathcal{F}_r$ are $r$-cross $t$-intersecting. This leads to $T_j\subseteq\cap_{i\in[r]}(\cap\mathcal{F}_i)$, which contradicts that $\mathcal{F}_1,\ldots,\mathcal{F}_r$ are non-trivial. Hence $\tau_t(\mathcal{G}_j)\geq t+1$ for  $1\leq j\leq r$. 

(i)\;Let $j\in[r]$ and let $T$ be a $t$-cover of $\mathcal{G}_j$. We may assume that there exists  $a\in[r]\setminus\{j\}$ and $F_a\in\mathcal{F}_a$ such that $|T\cap F_a|=t$. Let $F\in\cup_{i\neq j,a}\mathcal{F}_i$. We have to prove  $|T\cap F|\geq t+1$. Using (\ref{equthm3case21}), we derive that $T\cap F_a\subseteq F$. Since $\tau_t(\mathcal{G}_j)\geq t+1$, there is an $F_a'\in\mathcal{F}_a$ with $T\cap F_a\nsubseteq F_a'$. Then $T\cap F_a\subsetneqq T\cap F$ as $|T\cap F_a'\cap F|\geq t$, and thus $|T\cap F|\geq t+1$, as desired.

(ii)\;Let $j\in[r]$. Note that $\mathcal{G}_j$ contains some partition $G_j$ with exactly $t$ blocks since $s_j=t$, and then it is contained in each partition in  $\mathcal{F}_j$ as $\mathcal{F}_1,\ldots,\mathcal{F}_r$ are $r$-cross $t$-intersecting. Thus $\tau_t(\mathcal{F}_j)=t$.  Also note that $\tau_t(\mathcal{G}_j)\geq t+1$. Then it remains to prove $\tau_t(\mathcal{G}_j)\neq t+1$. We prove by contradiction. Assume that $\tau_t(\mathcal{G}_j)=t+1$, and let $T$ be a $t$-cover of $\mathcal{G}_j$ with $t+1$ blocks. Then we obtain $G_j\subseteq T$ from (\ref{equthm3case21}). From (i), there exists $a\in[r]\setminus\{j\}$ such that $|T\cap F|\geq t+1$, or equivalently $T\subseteq F$ for all $F\in\cup_{i\neq{j,a}}\mathcal{F}_i$. Now $G_j\subseteq\cup_{i\neq a}(\cap\mathcal{F}_i)$, namely, it forms a $t$-cover of $\mathcal{G}_a$ with only $t$ blocks, which is impossible as $\tau_t(\mathcal{G}_a)>t$.
\end{proof}
\begin{lemma}\label{lemmas_a=t}
Suppose $r\geq3$, $k_1\geq k_2\geq\cdots\geq k_r\geq t+2$ and $n\geq5-k_1+2k_1\log_2k_1$. Let   $\mathcal{F}_1\subseteq\spn{[n]}{k_1}$, $\mathcal{F}_2\subseteq\spn{[n]}{k_2}$, $\ldots,\;\mathcal{F}_r\subseteq\spn{[n]}{k_r}$ be non-trivial $r$-cross $t$-intersecting. If  $s_i=t$ for all $i\in[r]$, then  $\prod_{i=1}^r|\mathcal{F}_i|<\varphi(t+1,r)$.
\end{lemma}
\begin{proof}
Recall that our key ingredient, Lemma~\ref{lemmakey}, suggests bounding the size of a family using a suitable collection of $t$-covers. Here the desired families of $t$-covers are nothing but $\mathcal{G}_1,\ldots,\mathcal{G}_r$. For $i\in[r]$, we write  $m_i=\tau_t(\mathcal{G}_i)$, and fix a $G_i\in\mathcal{G}_i$ with $|G_i|=s_i=t$. Note that $G_i\subseteq\cap\mathcal{F}_i$, and Lemma \ref{lemmathm33} (ii) gives $m_i\geq t+2$ whenever $i\in[r]$. Note also that every member of $\mathcal{F}_i$ is a $t$-cover of $\mathcal{G}_i$, which gives $m_i\leq k_i$. Let $\ell_i=\max\{|G|:G\in\mathcal{G}_i\}$, then clearly $\ell_i\leq\min\{k_j:j\in[r]\setminus\{i\}\}\leq k_{r-1}$. By combining these with Lemma \ref{lemmakey}, we obtain that 
\begin{align}
	|\mathcal{F}_i|&=|(\mathcal{F}_i)_{G_i}|\leq\max\left\{(\ell_i-t+1)^{m_i-t}\spn{n-m_i}{k_i-m_i},\;(\ell_i-t+1)^{k_i-t}\right\}\nonumber\\
	&\leq  g(m_i,k_i,\ell_i,t,n)\binom{m_i}{t}^{-1},\;i=1,2,\ldots,r.\label{equthm3case211}
\end{align}
\noindent{\bf Case 1.}\;$r=3$.

Fix an $a\in[3]$ and a $t$-cover $T$ of $\mathcal{G}_a$ with size $m_a$. By Lemma \ref{lemmathm33} (i), there is another index $b$ such that $|T\cap F|\geq t+1$ for all $F\in\mathcal{F}_b$. Let $c$ be the remaining index. 

We use (\ref{equthm3case211}) to bound the size of $\mathcal{F}_a$, and let us consider $\mathcal{F}_b$ and $\mathcal{F}_c$, respectively.  Write $s=|T\cap G_b|$. Every partition in $\mathcal{F}_b$ contains $G_b$, and shares at least $t+1-s$ blocks with $T\setminus G_b$, implying that $\mathcal{F}_b={\textstyle\cup_H}\mathcal{F}_{G_b\cup H}$, where the summation is taken over all $(t+1-s)$-subsets $H$ of $T\setminus G_b$. Applying Lemma \ref{lemmamono'} to $k=k_b,u=s-1,\ell=m_a-1$ and $i=0$, we derive that
\begin{align*}
	|\mathcal{F}_b|\leq&{\textstyle\sum_H}|\mathcal{F}_{G_b\cup H}|\leq\binom{m_a-s}{t+1-s}\spn{n-t-(t+1-s)}{k_b-t-(t+1-s)}\leq(m_a-t)\spn{n-t-1}{k_b-t-1}.
\end{align*}
Now let us consider $\mathcal{F}_c$. If $T$ is also a $(t+1)$-cover of $\mathcal{F}_c$, then we can prove  $|\mathcal{F}_c|\leq(m_a-t)\spn{n-t-1}{k_c-t-1}$ by the same argument. Suppose that $|T\cap F_c|=t$ for some $F_c\in\mathcal{F}_c$. By (\ref{equthm3case21}), we obtain $T\cap F_c\subseteq F$ whenever $F\in\mathcal{F}_b$. Since $\tau_t(\mathcal{G}_a)>t$, there exists $F_c'\in\mathcal{F}_c$ with $T\cap F_c\nsubseteq F_c'$. This together with $T\cap G_c\subseteq T\cap F_c\cap F_c'$ yields  $s':=|T\cap G_c|<t$. Now every member of $\mathcal{F}_c$ contains at least $t-s'$ blocks of $T\setminus G_c$, and therefore
\begin{equation*}
	|\mathcal{F}_c|\leq\binom{m_a-s'}{t-s'}\spn{n-t-(t-s')}{k_c-t-(t-s')}\leq(m_a-t+1)\spn{n-t-1}{k_c-t-1},
\end{equation*} 
where in the second step we use Lemma \ref{lemmamono'} again.

Using our upper bounds above, to establish $|\mathcal{F}_a||\mathcal{F}_b||\mathcal{F}_c|<\varphi(t+1,r)$, we need only to verify, from the fact $\varphi(t+1,a)\leq\varphi(t+1,r)$, that
\begin{equation*}
	h(t+1,k_a,t,n)>(m_a-t)(m_a-t+1)g(m_a,k_a,\ell_a,t,n)\binom{m_a}{t}^{-1}.
\end{equation*}
This is proved in Lemma \ref{lemmathm3case21}.\\
{\bf Case 2.}\;$r\geq4$.

Intuitively, let us consider whether there exist $r-2$ families with a stronger intersection property. For $a\neq b\in[r]$, set $\mathcal{G}_{a,b}=\{\cap_{i\neq a,b}F_i:F_i\in\mathcal{F}_i,i\in[r]\setminus\{a,b\}\}$ and $s_{a,b}=\min\{|G|:G\in\mathcal{G}_{a,b}\}$. Note that $\{\mathcal{F}_i\}_{i\neq a,b}$ are $(r-2)$-cross $s_{a,b}$-intersecting.

When $s_{a,b}=t$ for all $a\neq b\in[r]$, by picking $G_{a,b}\in\mathcal{G}_{a,b}$ with $|G_{a,b}|=t$, we obtain that $G_{a,b}\subseteq F$ for all $F\in\mathcal{F}_a\cup\mathcal{F}_b$\;($a\neq b\in[r]$). It follows that $T_i:=\cup_{j\neq i}G_{i,j}\subseteq\cap\mathcal{F}_i,i=1,2,\ldots,r$. Since $\mathcal{F}_1,\ldots,\mathcal{F}_r$ are non-trivial, $|T_i|\geq t+1$ for all $i\in[r]$, and hence $$\prod_{i=1}^r|\mathcal{F}_i|\leq\prod_{i=1}^r\spn{n-t-1}{k_i-t-1}<\varphi(t+1,r).$$

It remains to consider the case that $s_{a,b}\geq t+1$ for some $a\neq b\in[r]$. We start by proving that $s_{a,b}<\min\{k_i:i\in[r]\setminus\{a,b\}\}$. Otherwise, there exists $G_{a,b}\in\spn{[n]}{s_{a,b}}$ such that  $\mathcal{F}_i=\{G_{a,b}\}$ for every $i\in[r]\setminus\{a,b\}$. Then $G_b$ is a subset of $G_{a,b}$, and so it is a $t$-cover of $\mathcal{G}_a$, which contradicts that $\tau_t(\mathcal{G}_a)\geq t+2$. Hence $s_{a,b}<\min\{k_i:i\in[r]\setminus\{a,b\}\}$.

If $s_{a,b}\geq t+2$, then Theorem \ref{thm1} gives $\prod_{i\neq a,b}|\mathcal{F}_i|\leq\prod_{i\neq a,b}\spn{n-t-2}{k_i-t-2}$, and this together with (\ref{equthm3case211}) yields
$$\prod_{i=1}^r|\mathcal{F}_i|\leq\left(\prod_{i\neq a,b}\spn{n-t-2}{k_i-t-2}\right)\left(\prod_{j=a,b}g(m_j,k_j,\ell_j,t,n)\binom{m_j}{t}^{-1}\right).$$
Then we get $\prod_{i=1}^r|\mathcal{F}_i|<\varphi(t+1,r)$ from Lemma \ref{lemmathm3case22}.

Suppose $s_{a,b}=t+1$. We claim that  $\mathcal{F}_a$ and $\mathcal{F}_b$ are cross $(t+1)$-intersecting. To the contrary, assume that for some $W_a\in\mathcal{F}_a$ and $W_b\in\mathcal{F}_b$, the partition $X:=W_a\cap W_b$ has exactly $t$ blocks. Let $T\in\mathcal{G}_{a,b}$ with $|T|=t+1$. Then  $X\subseteq F$ for all $F\in\cup_{i\neq a,b}\mathcal{F}_i$, and particularly $X\subseteq T$. Note that  $W_a\cap W_b\cap T=X$, then at least one of $W_a$ and $W_b$ does not contain $T$. Without loss of generality, suppose $T\nsubseteq W_a$. It follows from $|W_a\cap T|\geq t$ and $|T|=t+1$ that $T\cap W_a=X$, and hence $X\subseteq\cap\mathcal{F}_b$. However, this yields that $X$ is a $t$-cover of $\mathcal{G}_a$ as $X\subseteq\cap_{i\neq a}(\cap\mathcal{F}_i)$, which contradicts that $\tau_t(\mathcal{G}_a)\geq t+2$. Thus our claim is true. By applying  Theorem \ref{thm1} to $\{\mathcal{F}_a,\mathcal{F}_b\}$ and to $\{\mathcal{F}_i\}_{i\neq a,b}$, respectively, we deduce again $\prod_{i=1}^r|\mathcal{F}_i|\leq\prod_{i=1}^r\spn{n-t-1}{k_i-t-1}$, and thus $\prod_{i=1}^r|\mathcal{F}_i|<\varphi(t+1,r)$. This finishes the proof.
\end{proof}
\noindent{\bf Proof of Theorem \ref{thm3}.}\;This follows from Lemmas \ref{lemmas_a>t} and \ref{lemmas_a=t}.{\hfill$\square$}
\section{Technical estimates}\label{sectionappendix}
	In this section, we collect and prove several estimates required for this paper. Let us note that the functions  $f(m,k,\ell,t,n)$ and $g(m,k,\ell,t,n)$ are defined in (\ref{equfunf}) and (\ref{equfung}), respectively. The functions $r(n,k,\ell,t),r_1(n,k,\ell,t),r_2(n,k,\ell,t),h(m,k,t,n)$ and $\varphi(m,a)$ used in Sections \ref{section3} and \ref{sectionrcross} are defined in (\ref{equfunr}), (\ref{equfunr1}), (\ref{equfunr2}), (\ref{equfunh}) and (\ref{equfunphi}), respectively.
\begin{lemma}\label{lemmaestmtspn}
	Let $r\geq2,c\geq1$ and $m\geq cr(1+\ln r)$. Then $\left(1-r(er)^{-c}\right)\frac{r^m}{r!}<\spn{m}{r}<\frac{r^m}{r!}$. In particular, we have $\spn{m}{r}>\frac{1}{4}r^{m-r+2}$.
\end{lemma}
\begin{proof}
	Recall that $r!\spn{m}{r}$ counts the number of surjections of $[m]$ onto $[r]$. Then clearly $r!\spn{m}{r}<r^m$. For the lower bound, note that for each  $i\in[r]$, there are $(r-1)^m$ maps sending no element to $i$. Applying the  union bound yields 
	$$r!\spn{m}{r}\geq r^m-r(r-1)^m=r^m\left(1-r\left(\frac{r-1}{r}\right)^m\right),$$
	and then the desired bound holds as  $\left(\frac{r-1}{r}\right)^m<e^{-\frac{m}{r}}\leq(er)^{-c}$. Further, we infer from $r!\leq2r^{r-2}$ that $\spn{m}{r}>(1-e^{-1})r^{m}/r!>\frac{1}{4}r^{m-r+2}$.
\end{proof}
	\begin{lemma}[\cite{Wen-Lv-2026}]\label{lemmaqkt}
		Let $t\geq1$. Then $Q(s,t):=s^{L(s+t,t)-2s-t+1}\binom{s+t}{t}^{-1}$ is increasing on $s\geq2$. In particular, $Q(s,t)\geq18.$
	\end{lemma}
\begin{lemma}\label{lemmamono''}
Let $k\geq u+2\geq t+2$, $\ell\geq t$ and $n\geq L(\max\{k,\ell\},t)$. Then
\begin{equation*}
	\binom{u}{t}\spn{n-u}{k-u}>\binom{k}{t}(\ell-t+1)^{k-u}(k-u).
\end{equation*}
\end{lemma}
\begin{proof}
Clearly, if we verify the inequality above for $\ell=k$ and $n\geq L(k,t)$, then it also holds provided that $\ell<k$ and $n\geq L(k,t)$. Hence it suffices to assume $\ell\geq k$. We claim that 
\begin{equation}\label{equlemmamonon2}
	L(\ell,t)\geq L(k-u+t,t)+3(u-t)+(k-u)\log_{(k-u)}\frac{\ell-t+1}{k-t+1}.
\end{equation}
For each $s\geq k-u+t+1$, we have $s\geq t+3$, and then
\begin{equation}\label{equlktlk-1t}
	L(s,t)-L(s-1,t)=\log_2((t+1)(s-t+1)(1+(s-t)^{-1})^{s-t})>3.
\end{equation}
Hence (\ref{equlemmamonon2}) holds trivially for $\ell=k$. Suppose $\ell>k$, then we have
\begin{align*}
L(\ell,t)-L(k,t)&>(\ell-t+1)\log_2(\ell-t+1)-(k-t+1)\log_2(k-t+1)\\&>(\ell-k)\log_2(k-t+1)>\frac{\ell-k}{\ln2}.
\end{align*}
On the other hand, we have 
$$\frac{k-u}{\ln(k-u)}\cdot\ln\frac{\ell-t+1}{k-t+1}=\frac{k-u}{\ln(k-u)}\cdot\ln\left(1+\frac{\ell-k}{k-t+1}\right)<\frac{k-t}{\ln2}\cdot\frac{\ell-k}{k-t+1}<\frac{\ell-k}{\ln2}.$$
It follows that $L(\ell,t)\geq L(k,t)+\frac{k-u}{\ln(k-u)}\cdot\ln\frac{\ell-t+1}{k-t+1}$, and then our claim follows as (\ref{equlktlk-1t}) gives $L(k,t)> L(k-u+t,t)+3(u-t)$. 

Now we derive from Lemma \ref{lemmaestmtspn} and  (\ref{equlemmamonon2}) that 
\begin{align*}
	\spn{n-u}{k-u}&>\frac{1}{4}(k-u)^{n-k+2}\geq\frac{1}{4}(k-u)^{L(k-u+t,t)+3(u-t)-k+2}\left(\frac{\ell-t+1}{k-t+1}\right)^{k-u}\\
	&=\frac{1}{4}Q(k-u,t)\binom{k-u+t}{t}(k-u)^{2(u-t)+1}\left(\frac{k-u}{k-t+1}\right)^{k-u}(\ell-t+1)^{k-u}\\
	&>4.5\binom{k-u+t}{t}\cdot4^{u-t}e^{-(u-t+1)}(\ell-t+1)^{k-u}(k-u).
\end{align*}
where the expression $Q$ is defined in Lemma \ref{lemmaqkt}. To proceed, we need the function $\binom{x-a}{t}\binom{x}{t}^{-1}$ for $a\geq0$ and $x\geq t$. One can easily check that it is increasing as $x\geq t$ increases. By applying this to $a=u-t$ and $a=2$, respectively, we have $\binom{k-u+t}{t}\binom{k}{t}^{-1}\binom{u}{t}\geq\binom{t+2}{t}\binom{u+2}{t}^{-1}\binom{u}{t}\geq1$, and thus $\binom{k-u+t}{t}\binom{u}{t}\geq\binom{k}{t}$. It follows that $\binom{u}{t}\spn{n-u}{k-u}>\binom{k}{t}(\ell-t+1)^{k-u}(k-u)$, as desired.
\end{proof}
	\begin{lemma}\label{lemmamono}
		Let $k\geq t+2$, $\ell\geq t$ and $n\geq L(\max\{k,\ell\},t)$. The following hold.
		\begin{itemize}
			\item[\rm(i)]The functions  $(\ell-t+1)^{m-t}\spn{n-m}{k-m}$ and  $f(m,k,\ell,t,n)$ are both strictly decreasing as $m\in[t,k-1]$ increases.
			\item[\rm(ii)]Fix a $u\geq t$. If $k\geq u+2$, then $g(m,k,\ell,t,n)\leq f(u,k,\ell,t,n)$ for $u\leq m\leq k$, with equality only if $m=u$.
			\item[\rm(iii)]We have $g(m,k,\ell,t,n)\leq\spn{n-t}{k-t}$
			for $t\leq m\leq k$, with equality only if $m=t$.
		\end{itemize}	
	\end{lemma}
	\begin{proof}
		{\rm(i)}\;For each $m\in[t,k-2]$, we infer from Lemma \ref{lemmaspn2lkt} (i) that 
		\begin{align*}
			Q_m&:=\dfrac{f(m+1,k,\ell,t,n)}{f(m,k,\ell,t,n)}=\dfrac{(\ell-t+1)(m+1)}{m-t+1}\cdot\frac{\spn{n-m-1}{k-m-1}}{\spn{n-m}{k-m}}\\
			&<(\ell-t+1)(t+1)\left((t+1)(\max\{k,\ell\}-t+1)\right)^{-1}\leq1.
		\end{align*}
		Hence $f(m,k,\ell,t,n)$ is strictly decreasing as $m\in[t,k-1]$ increases. For the other one, since  $$\dfrac{(\ell-t+1)^{m+1-t}\spn{n-m-1}{k-m-1}}{(\ell-t+1)^{m-t}\spn{n-m}{k-m}}=\dfrac{(m-t+1)Q_m}{m+1}<1$$
		 for $t\leq m\leq k-2$, it is also strictly decreasing with $m$.
		 
		 {\rm(ii)}\;Note that (i) gives  $g(m,k,\ell,t,n)\leq\max\{f(u,k,\ell,t,n), f(k,k,\ell,t,n)\}$ for every $m$ with $u\leq m\leq k-1$. Hence it suffices to prove $f(u,k,\ell,t,n)>f(k,k,\ell,t,n)$, or equivalently $\binom{u}{t}\spn{n-u}{k-u}>\binom{k}{t}(\ell-t+1)^{k-u}$. This follows from Lemma \ref{lemmamono''}.
		 
(iii)\;This follows immediately by applying (ii) to $u=t$.
\end{proof}
\begin{lemma}\label{lemmamono'}
Let $k\geq t+2$, $\ell\geq t+1$ and $n\geq L(\max\{k,\ell\},t)$. Then $\binom{\ell-i-u}{t-u}\spn{n-t-i-(t-u)}{k-t-(t-u)}\leq(\ell-i-t+1)\spn{n-t-i-1}{k-t-1}$ whenever $u\in\{0,1,\ldots,t-1\}$ and  $i\in\{0,1\}$.
\end{lemma}
\begin{proof}
Write $\psi(u):=\binom{\ell-i-u}{t-u}\spn{(n-i)-t-(t-u)}{k-t-(t-u)},\;0\leq u\leq t-1$. Note that $n-i\geq L(\max\{k,\ell\},t)-1$, then from Lemma \ref{lemmaspn2lkt} (i), we have
$$\psi(s)\geq\frac{t-s+1}{\ell-i-s+1}\cdot(t+1)(\ell-t+1)\psi(s-1)\geq\frac{2(t+1)(\ell-t+1)}{\ell-t+2}\psi(s-1)\geq\psi(s-1)$$
for $s=1,2,\ldots,t-1$. It follows that $\psi(u)\leq\psi(t-1)$, which is nothing but the inequality to be verified.
\end{proof}
\begin{lemma}\label{lemmaw}
Suppose $k\geq t+2$ and $n\geq L(k,t)$. For a partition $G$, set  $\mathcal{W}(G):=\left\{F\in\spn{[n]}{k}:|F\cap G|\geq t\right\}$. Then the following hold. 
\begin{itemize}
\item[\rm(i)]If $G\in\spn{[n]}{t+2}$, then $|\mathcal{W}(G)|\leq h(t+1,k,t,n)$. Moreover, if $G$ contains at most $t$ singletons, then $|\mathcal{W}(G)|<(t+0.6)\spn{n-t}{k-t}.$
\item[\rm(ii)]For all partition $G$ with $|G|=s\in[t+2,k]$, we have $|\mathcal{W}(G)|\leq6h(s,k,t,n)$.
\end{itemize}
\end{lemma}
\begin{proof}
(i)\;If $G$ consists of $t+1$ singletons and one block of size $n-t-1$, then every member of $\mathcal{W}$ contains all but at most one of these singletons, and hence $|\mathcal{W}|=h(t+1,k,t,n)=(t+1)\spn{n-t}{k-t}-t\spn{n-t-1}{k-t-1}$. From Lemma  \ref{lemmaspn2lkt} (i), we have $h(t+1,k,t,n)>(t+0.6)\spn{n-t}{k-t}$. Therefore, it remains to verify that, if $G$ has at most $t$ singletons, then $|\mathcal{W}|<(t+0.6)\spn{n-t}{k-t}$. 

We proceed by considering the function   $u(\mathbf{x}):=\sum_{i<j}\spn{x_i+x_j}{k-t}$ in $D$, where $D$ is the set of decreasing positive integer sequences $\mathbf{x}=(x_1,\ldots,x_{t+2})$ satisfying $x_2\geq2$ and $\sum_{i}x_i=n$. Let $\mathbf{b}=(b_1,\ldots,b_{t+2})$ be a maximum point of $u(\mathbf{x})$ in $D$ with the largest leading term. We claim that  $\mathbf{b}=(n-t-2,2,1,\ldots,1)$. Indeed,  we have first $b_{t+2}=1$. Otherwise, assume that $b_{t+2}\geq2$. Then $\mathbf{b}':=(b_1+1,b_2,\ldots,b_{t+1},b_{t+2}-1)$ also lies in $D$. However, we derive from (\ref{equrecurrence}) and $b_1\geq b_{t+2}$ that 
\begin{align*}
	u(\mathbf{b}')-u(\mathbf{b})=&\sum_{1<j<t+2}\left(\spn{b_1+b_j+1}{k-t}+\spn{b_j+b_{t+2}-1}{k-t}-\spn{b_1+b_j}{k-t}-\spn{b_j+b_{t+2}}{k-t}\right)\\
	=&\sum_{1<j<t+2}\left(\spn{b_1+b_j}{k-t-1}-\spn{b_j+b_{t+2}-1}{k-t-1}\right)\\
	&+(k-t-1)\sum_{1<j<t+2}\left(\spn{b_1+b_j}{k-t}-\spn{b_j+b_{t+2}-1}{k-t}\right)\geq0.
\end{align*}
This contradicts the choice of $\mathbf{b}$. Hence $b_{t+2}=1$, and then we have  $b_{t+1}=\cdots=b_{3}=1$ and $b_2=2$ by using the same argument repeatedly. Thus our claim is true.

Suppose $G$ has at most $t$ singletons, and write $G=\{B_1,B_2,\ldots,B_{t+2}\}$. Then 
 \begin{align*}
|\mathcal{W}|&\leq\sum_{i<j}|\mathcal{W}_{G\setminus\{B_i,B_j\}}|\leq\sum_{i< j}\spn{|B_i|+|B_j|}{k-t}\\
&\leq u(\mathbf{b})=\spn{n-t}{k-t}+t\spn{n-t-1}{k-t}+t\spn{3}{k-t}+\binom{t}{2}\spn{2}{k-t}\\
&<\left(1+\frac{t}{2}\right)\spn{n-t}{k-t}+3t+\binom{t}{2}<(t+0.6)\spn{n-t}{k-t}.
 \end{align*}
To see the last inequality, note that Lemma \ref{lemmamono''} gives $\spn{n-t}{k-t}>18\binom{t+2}{2}$. Then it is easy to check that  $(0.5t-0.4)\spn{n-t}{k-t}>(9t-8)\binom{t+2}{2}\geq3t+\binom{t}{2}$. This finishes the proof of (i).

(ii)\;First, the union bound gives $|\mathcal{W}(G)|\leq\sum_{X\in\binom{G}{t}}\spn{n-|\cup X|}{k-t}\leq\binom{s}{t}\spn{n-t}{k-t}$. To proceed, we lower bound $h(s,k,t,n)$. Recall from (\ref{equfunh}) that $h(s,k,t,n)$ counts the size of $\left\{F\in\spn{[n]}{k}:|F\cap[[s]]|\geq t\right\}$, which is the union of the families of $k$-partitions of $[n]$  containing $X$ with $X$ ranging over the $t$-subsets of $[[s]]$. For $1\leq i\leq t$, there are $\binom{t}{t-i}\binom{s-t}{i}\binom{s}{t}/2$ unordered pairs $\{X,Y\}$ of $t$-subsets of $[[s]]$ with $|X\cup Y|=t+i$. By combining these with Lemma \ref{lemmain-ex}, we obtain that
\begin{align*}
h(s,k,t,n)\geq\binom{s}{t}\spn{n-t}{k-t}-\frac{1}{2}\binom{s}{t}\sum_{i=1}^{t}\binom{t}{t-i}\binom{s-t}{i}\spn{n-t-i}{k-t-i}.
\end{align*}
Form Lemma \ref{lemmaspn2lkt} (i), we have $\spn{n-t}{k-t}>(t(k-t))^i\spn{n-t-i}{k-t-i}\geq(i!)^2\binom{t}{i}\binom{s-t}{i}\spn{n-t-i}{k-t-i}$ for $1\leq i\leq t$. It follows that
\begin{align*}
h(s,k,t,n)\binom{s}{t}^{-1}\spn{n-t}{k-t}^{-1}\geq1-\frac{1}{2}\sum_{i=1}^{t}\frac{1}{(i!)^2}>1-\frac{1}{2}\cdot\frac{\pi^2}{6}>\frac{1}{6},
\end{align*}
and thus $h(s,k,t,n)\geq\frac{1}{6}\binom{s}{t}\spn{n-t}{k-t}\geq\frac{1}{6}|\mathcal{W}(G)|$, as desired.
\end{proof}
\begin{lemma}\label{lemmar2kllk}
We have $r_2(n,k,\ell,t)>r_2(n,\ell,k,t)$ for $n>k>\ell\geq t+2$.
\end{lemma}
\begin{proof}
Let us write $h(m,a):=\spn{m}{a}/\spn{m-1}{a-1}$ for $2\leq a\leq m$. Clearly $h(m,a)-h(m,a-1)$ equals $\spn{m-1}{a-1}^{-1}\spn{m-1}{a-2}^{-1}\left(\spn{m}{a}\spn{m-1}{a-2}-\spn{m}{a-1}\spn{m-1}{a-1}\right)$. From (\ref{equrecurrence}) and Lemma \ref{lemmalogconcavity} we infer that $\spn{m}{a}\spn{m-1}{a-2}-\spn{m}{a-1}\spn{m-1}{a-1}$ is equal to
\begin{align*}
a\spn{m-1}{a}\spn{m-1}{a-2}-(a-1)\spn{m-1}{a-1}^2&\leq\left(\frac{a(a-2)}{a-1}-(a-1)\right)\spn{m-1}{a-1}^2\\
&=-\frac{1}{a-1}\spn{m-1}{a-1}^2<0.
\end{align*}
It follows that $h(m,a)<h(m,a-1)$ for $2\leq a\leq m$, and hence 
\begin{equation*}\label{equhma}
\spn{m}{a}\spn{m-1}{b-1}<\spn{m-1}{a-1}\spn{m}{b}\;\mbox{for}\;2\leq b<a\leq m.
\end{equation*}
Therefore,
\begin{align*}
r_2(n,k,\ell,t)-r_2(n,\ell,k,t)&=(t+1)\left(\spn{n-t-1}{k-t-1}\spn{n-t}{\ell-t}-\spn{n-t}{k-t}\spn{n-t-1}{\ell-t-1}\right)>0
\end{align*}
as $\ell<k$.
\end{proof}
\begin{lemma}\label{lemmar1r2}
Suppose $k\geq\ell\geq t+2$ and $n\geq2L(k,t)$. The following hold.
\begin{itemize}
\item[\rm(i)]If $\ell\geq2t+2$, then $r_1(n,k,\ell,t)>r_2(n,k,\ell,t)$.
\item[\rm(ii)]If $\ell\leq2t+1$ and $(k,\ell)\neq(2t+1,2t+1)$ or $(4,3)$, then  $r_1(n,k,\ell,t)<r_2(n,k,\ell,t)$.
\end{itemize}
\end{lemma}
\begin{proof}
Write $D=r_1(n,k,\ell,t)-r_2(n,k,\ell,t)$. From Lemmas \ref{lemmain-ex} and  \ref{lemmaspn2lkt} (ii), we have 
\begin{equation}\label{equfunr1lbub}
\ell-t-\frac{1}{(t+1)^2}<\frac{r_1(n,k,\ell,t)}{\spn{n-t-1}{k-t-1}\left(\spn{n-t}{\ell-t}+t\right)}<\ell-t
\end{equation}
and
$$t+1-\frac{1}{(t+1)(\ell-t+1)^2}<\frac{r_2(n,k,\ell,t)}{\spn{n-t-1}{k-t-1}\spn{n-t}{\ell-t}}<t+1.$$

If $\ell\geq2t+2$, then $D\spn{n-t-1}{k-t-1}^{-1}\spn{n-t}{\ell-t}^{-1}>\ell-2t-1-\frac{1}{(t+1)^2}>0$ and thus $r_1(n,k,\ell,t)>r_2(n,k,\ell,t)$. If $\ell\leq2t$, then $D\spn{n-t-1}{k-t-1}^{-1}<(\ell-2t-1+\frac{1}{2})\spn{n-t}{\ell-t}+t(\ell-t)\leq t(\ell-t)-\frac{1}{2}\spn{n-t}{\ell-t}$, and so $D<0$ as Lemma \ref{lemmamono''} gives $\spn{n-t}{\ell-t}>2t(\ell-t)$.

It remains to prove $D<0$ provided that $k>\ell=2t+1$ and $(k,\ell)\neq(4,3)$. We proceed by  estimating the Stirling partition numbers involved. It is readily seen that $n-t\geq2L(k,t)-t>2(k-t)(1+\log_2(k-t))$. Then from Lemma \ref{lemmaestmtspn}, we derive 
$$\left(1-\frac{1}{2e^2}\right)\frac{(\ell-t)^{n-t}}{(\ell-t)!}<\spn{n-t}{\ell-t}<\frac{(\ell-t)^{n-t}}{(\ell-t)!}$$
and
$$\left(1-\frac{1}{2e^2}\right)\frac{(\ell-t-1)^{n-t-1}}{(\ell-t-1)!}<\spn{n-t-1}{\ell-t-1}\leq\frac{(\ell-t-1)^{n-t-1}}{(\ell-t-1)!}.$$
We note that the relation above holds trivially for $\ell=t+2$. Write $\alpha=1-\frac{1}{2e^2}$ for short. Note that $\alpha>0.9$. It follows that $\spn{n-t}{\ell-t}\spn{n-t-1}{\ell-t-1}^{-1}>\alpha\left(\frac{\ell-t}{\ell-t-1}\right)^{n-t-1}$. Similarly, we have $\spn{n-t-1}{k-t-1}\spn{n-t-2}{k-t-2}^{-1}<\alpha^{-1}\left(\frac{k-t-1}{k-t-2}\right)^{n-t-2}$ for $k\geq t+3$. Suppose first $t=1,\ell=3$ and $k\geq5$. Now a direct computation gives $D=3\spn{n-2}{k-2}-2^{n-2}\spn{n-3}{k-3}$, and thus
\begin{align*}
	D\spn{n-3}{k-3}^{-1}<3\alpha^{-1}\left(\frac{k-2}{k-3}\right)^{n-3}-2^{n-2}\leq3\alpha^{-1}\cdot(1.5)^{n-3}-2^{n-2}<0
\end{align*}
as $n\geq2L(k,t)=4+2k(\log_2k+1)>10$.

 Suppose $t\geq2$. Write $r_0=r_0(n,k,\ell,t)=\sum_{j=2}^{\ell-t}(-1)^{j}\binom{\ell-t}{j}\spn{n-t-j}{k-t-j}$ for short. Note that $r_1(n,k,\ell,t)=\left((\ell-t)\spn{n-t-1}{k-t-1}-r_0\right)\left(\spn{n-t}{\ell-t}+t\right)$. Now Lemmas \ref{lemmain-ex} and \ref{lemmaspn2lkt} (ii), and $k>2t+1$ yield $r_0>\left(\binom{t+1}{2}-0.1\right)\spn{n-t-2}{k-t-2}$, and hence 
\begin{align*}
	D&=\left(t(\ell-t)+t\spn{n-t-1}{\ell-t-1}\right)\spn{n-t-1}{k-t-1}-r_0\left(\spn{n-t}{\ell-t}+t\right)\\
	&<(t+0.1)\spn{n-t-1}{k-t-1}\spn{n-t-1}{\ell-t-1}-\left(\binom{t+1}{2}-0.1\right)\spn{n-t-2}{k-t-2}\spn{n-t}{\ell-t}.
\end{align*}
To see the inequality above, note that $t\geq2$ guarantees $\ell=2t+1\geq t+3$, then by Lemma \ref{lemmaspn2lkt} (ii), we obtain that $\spn{n-t-1}{\ell-t-1}>(t+1)^2(\ell-t+1)^2>10t(\ell-t)$. These together with $k\geq\ell+1$ and $\alpha^2>0.9^2>0.8$ yield
\begin{align*}
\frac{\spn{n-t-2}{k-t-2}\spn{n-t}{\ell-t}}{\spn{n-t-1}{k-t-1}\spn{n-t-1}{\ell-t-1}}>\alpha^{2}\cdot\frac{k-t-1}{k-t-2}\left(\frac{\ell-t}{\ell-t-1}\cdot\frac{k-t-2}{k-t-1}\right)^{n-t-1}>0.8.
\end{align*}
Hence for $t\geq2$, the quantity  $D\spn{n-t-1}{k-t-1}^{-1}\spn{n-t-1}{\ell-t-1}^{-1}$ is less than 
\begin{align*}
	t+0.1-0.8\left(\binom{t+1}{2}-0.1\right)< t+0.2-0.4t(t+1)<0.
\end{align*}
This finishes the proof.
\end{proof}
\begin{lemma}\label{lemmar1kllk}
	If $k>\ell\geq t+2$ and $n\geq\max\{2L(k,t),t+1+(k-t)(\ell-t)\}$, then $r_1(n,k,\ell,t)>r_1(n,\ell,k,t)$.  
\end{lemma}
\begin{proof}
From Lemmas \ref{lemmain-ex} and  \ref{lemmaspn2lkt} (ii), we have 
\begin{align}
		\frac{r_1(n,k,\ell,t)}{r_1(n,\ell,k,t)}&>\frac{\ell-t-0.5}{k-t}\cdot\frac{\spn{n-t-1}{k-t-1}\left(\spn{n-t}{\ell-t}+t\right)}{\left(\spn{n-t}{k-t}+t\right)\spn{n-t-1}{\ell-t-1}}>\frac{\ell-t-0.5}{1.1(k-t)}\cdot\frac{\spn{n-t-1}{k-t-1}\spn{n-t}{\ell-t}}{\spn{n-t}{k-t}\spn{n-t-1}{\ell-t-1}},\label{equlemmar1kllk1}
	\end{align}
where the second inequality follows as Lemma \ref{lemmaspn2lkt} (ii) gives $\spn{n-t}{k-t}>10t$. By the same argument as in Lemma \ref{lemmar1r2}, we derive 
$$\alpha\left(\frac{j-t}{j-t-1}\right)^{n-t-1}<\spn{n-t}{j-t}\spn{n-t-1}{j-t-1}^{-1}<\alpha^{-1}\left(\frac{j-t}{j-t-1}\right)^{n-t-1}\;\mbox{for}\;j=k,\ell,$$
where $\alpha=1-\frac{1}{2e^2}$. These together with (\ref{equlemmar1kllk1}) and $\ell-t-0.5\geq\frac{3}{4}(\ell-t)$ yield 
	\begin{align*}
		\frac{r_1(n,k,\ell,t)}{r_1(n,\ell,k,t)}&>\alpha^2\cdot\frac{\ell-t-0.5}{1.1(k-t)}\left(\frac{k-t-1}{k-t}\cdot\frac{\ell-t}{\ell-t-1}\right)^{n-t-1}\\
		&\geq\frac{6(\ell-t)}{11(k-t)}\left(\frac{(\ell-t)^2}{(\ell-t)^2-1}\right)^{n-t-1}>\frac{6(\ell-t)}{11(k-t)}e^{\frac{n-t-1}{(\ell-t)^2}}\\
		&\geq\frac{6(\ell-t)}{11(k-t)}\cdot\frac{e(n-t-1)}{(\ell-t)^2}\geq\frac{6e}{11}>1,
	\end{align*}
 where the second inequality holds as $\alpha^2=\left(1-\frac{1}{2e^2}\right)^2>0.9^2>0.8$.
\end{proof}
\begin{lemma}\label{lemmafunhmono}
Suppose $r\geq3$, $k_1\geq k_2\geq\cdots\geq k_r\geq t+2$ and $n\geq L(k_1,t)$. Then $\varphi(m,a)\leq\varphi(t+1,r)$ for $m\geq t+1$ and $1\leq a\leq r$, with equality precisely if $m=t+1$ and $k_a=k_r$.
\end{lemma}
\begin{proof}
By a routine computation, 
\begin{align*}
	\varphi(t+1,a)-\varphi(t+1,r)=(r_2(n,k_r,k_a,t)-r_2(n,k_a,k_r,t))\prod_{i\neq a,r}\spn{n-t-1}{k_i-t-1}.
\end{align*}
It follows from Lemma \ref{lemmar2kllk} that $\varphi(t+1,a)\leq\varphi(t+1,r)$ and that equality holds implies $k_a=k_r$. It remains to prove $\varphi(m,a)<\varphi(t+1,a)$ for $m\geq t+2$. Note that  $h(m,k_a,t,n)\leq\sum_{X\in\binom{[m]}{t}}\spn{n-|\cup X|}{k_a-t}=\binom{m}{t}\spn{n-t}{k_a-t}$. This yields 
\begin{align*}
h(t+1,k_a,t,n)=(t+1)\spn{n-t}{k_a-t}-t\spn{n-t-1}{k_a-t-1}>t\binom{m}{t}^{-1}h(m,k_a,t,n), 
\end{align*}
and then we deduce from Lemma \ref{lemmaspn2lkt} (i) that
\begin{align*}
\frac{\varphi(t+1,a)}{\varphi(m,a)}&>t\binom{m}{t}^{-1}((t+1)(k_1-t+1))^{(r-1)(m-t-1)}\\
&\geq t\binom{t+2}{2}^{-1}(t+1)(k_1-t+1)>1.
\end{align*}
This completes the proof.
\end{proof}
\begin{lemma}\label{lemmathm3case21}
Suppose that $k\geq m\geq t+2,\;\ell\geq t$ and $n\geq L(\max\{k,\ell\},t)$. Then
\begin{equation}\label{equlemmathm3case21}
h(t+1,k,t,n)>(m-t)(m-t+1)g(m,k,\ell,t,n)\binom{m}{t}^{-1}.
\end{equation}
\end{lemma}
\begin{proof}
By Lemma \ref{lemmaspn2lkt} (i), we have $h(t+1,k,t,n)>(t+0.6)\spn{n-t}{k-t}$.
Suppose $m\leq k-2$, then Lemma \ref{lemmamono} (ii) gives $g(m,k,\ell,t,n)=f(m,k,\ell,t,n)=(\ell-t+1)^{m-t}\binom{m}{t}\spn{n-m}{k-m}$. On the other hand, we derive from Lemma \ref{lemmaspn2lkt} (i) that
\begin{align*}
\left(t+0.6\right)\spn{n-t}{k-t}\spn{n-m}{k-m}^{-1}&>1.6((t+1)(\ell-t+1))^{m-t}\\
&>(m-t)(m-t+1)(\ell-t+1)^{m-t},
\end{align*}
where in the second step we used $2^{x}\geq 2x(x+1)/3$ for $x\geq2$ (by setting $x=m-t$). Then (\ref{equlemmathm3case21}) holds.

Suppose $m\in\{k-1,k\}$. Now $g(m,k,\ell,t,n)=(\ell-t+1)^{k-t}\binom{k}{t}$, and then the right-hand side in (\ref{equlemmathm3case21})  is at most $(m-t)(\ell-t+1)^{k-t}\binom{k}{t}$ as $\binom{m}{t}\geq m-t+1$. Thus (\ref{equlemmathm3case21}) follows as Lemma \ref{lemmamono''} gives  $h(t+1,k,t,n)>1.6\binom{k}{t}(\ell-t+1)^{k-t}(k-t)$.
\end{proof}
\begin{lemma}\label{lemmathm3case22}
Let $r\geq4$, $k_1\geq k_2\geq\cdots\geq k_r\geq t+2$ and $n\geq L(k_1,t)$. Let $m_i\in[t+2,k_i]$ and $\ell_i\in[t,k_1]$ for  $i=1,2,\ldots,r$. Then for all $a\neq b\in[r]$, we have
$$\left(\prod_{i\neq a,b}\spn{n-t-2}{k_i-t-2}\right)\left(\prod_{j=a,b}g(m_j,k_j,\ell_j,t,n)\binom{m_j}{t}^{-1}\right)<\varphi(t+1,r).$$
\end{lemma}
\begin{proof} 
Denote by $L$ the expression of the left-hand side above. From Lemma \ref{lemmafunhmono}, it suffices to prove $L<\varphi(t+1,a)= h(t+1,k_a,t,n)\prod_{i\neq a}\spn{n-t-1}{k_i-t-1}$. From $r\geq4$ and Lemma \ref{lemmaspn2lkt} (i), we derive $\prod_{i\neq a,b}\spn{n-t-1}{k_i-t-1}\geq(t+1)^2(k_1-t+1)^2\prod_{i\neq a,b}\spn{n-t-2}{k_i-t-2}$. With  this inequality and Lemma \ref{lemmathm3case21}, which gives $h(t+1,k_a,t,n)>6g(m_a,k_a,\ell_a,t,n)\binom{m_a}{t}^{-1}$, it is sufficient to  establish
\begin{equation}\label{equlemmathm3}
	g(m_b,k_b,\ell_b,t,n)\binom{m_b}{t}^{-1}<6(t+1)^{2}(k_1-t+1)^2\spn{n-t-1}{k_b-t-1}.
\end{equation}
When $k_b=t+2$, we have $m_b=t+2$, and then $g(m_b,k_b,\ell_b,t,n)\binom{m_b}{t}^{-1}=(\ell_b-t+1)^2$. Thus (\ref{equlemmathm3}) holds trivially. Suppose $k_b\geq t+3$. Note that the right-hand side of (\ref{equlemmathm3}) is larger than $(t+1)(\ell_b-t+1)\spn{n-t-1}{k_b-t-1}=f(t+1,k_b,\ell_b,t,n)$, then consequently it is larger than $g(m_b,k_b,\ell_b,t,n)$ from Lemma \ref{lemmamono} (ii). Hence (\ref{equlemmathm3}) also holds. 
\end{proof}
	\section*{Acknowledgments}
	B. Lv is supported by National Natural Science Foundation of China (12571347 \& 12131011), and Beijing Natural Science Foundation (1252010).
	\
	\addcontentsline{toc}{chapter}{Bibliography}
	
	{
		}
\end{document}